\documentclass[10pt,reqno]{amsart}

\usepackage[a4paper, left=30mm,right=30mm,top=30mm,bottom=30mm,marginpar=25mm]{geometry}
\usepackage[english]{babel}
\usepackage[utf8x]{inputenc}
\usepackage{setspace}
\usepackage{amsmath}
\usepackage{mathtools}
\usepackage[normalem]{ulem}
\usepackage{amsfonts}
\usepackage{amssymb}
\usepackage{fancyhdr}
\usepackage{textcomp}
\usepackage{graphicx} 
\usepackage{float}
\usepackage{listings}
\usepackage{amssymb}
\usepackage{color}

\newtheorem{theorem}{Theorem}[section]
\newtheorem{lemma}{Lemma}[section]

\newtheorem{definition}{Definition}[section]
\newtheorem{remark}{Remark}[section]

\newcommand{\sit}{\frac{\partial \hat{\psi}}{\partial \xi^B}(\xi,\theta)\frac{\partial\Phi^B}{\partial F_{i\alpha}}(F)}
\newcommand{\sitb}{\frac{\partial \hat{\psi}}{\partial \xi^B}(\bar{\xi},\bar{\theta})\frac{\partial\Phi^B}{\partial F_{i\alpha}}(\bar{F})}
\newcommand{\ph}{\frac{\partial\Phi^B}{\partial F_{i\alpha}}(F)}
\newcommand{\phv}{\frac{\partial\Phi^B}{\partial F_{i\alpha}}(F)v_i}
\newcommand{\phb}{\frac{\partial\Phi^B}{\partial F_{i\alpha}}(\bar{F})}
\newcommand{\phbv}{\frac{\partial\Phi^B}{\partial F_{i\alpha}}(\bar{F})\bar{v}_i}
\newcommand{\pg}{\frac{\partial \hat{\psi}}{\partial \xi^B}(\xi,\theta)}
\newcommand{\pgb}{\frac{\partial \hat{\psi}}{\partial \xi^B}(\bar{\xi},\bar{\theta})}

\newcommand{\cof}{\hbox{cof}\,}
\newcommand{\dt}{\hbox{det}\,}

\definecolor{listinggray}{gray}{0.9}
\definecolor{lbcolor}{rgb}{0.9,0.9,0.9}

\def\del{\partial}
\allowdisplaybreaks
\begin{document}

\numberwithin{equation}{section}

	\title[Polyconvex thermoviscoelasticity and Applications]{A symmetrizable extension of polyconvex thermoelasticity 
	and applications  to zero-viscosity limits and weak-strong uniqueness}

	\author[C. Christoforou]{Cleopatra Christoforou}
	\address[Cleopatra Christoforou]{Department of Mathematics and Statistics,
 University of Cyprus, Nicosia 1678, Cyprus.}
	\email{christoforou.cleopatra@ucy.ac.cy} 
	\author[M. Galanopoulou]{Myrto Galanopoulou}
	\address[Myrto Galanopoulou]{Computer, Electrical, Mathematical Sciences \& Engineering Division, King Abdullah University of Science and Technology (KAUST), Thuwal, Saudi Arabia.}
	\email{myrtomaria.galanopoulou@kaust.edu.sa}
	\author[A. E. Tzavaras]{Athanasios E. Tzavaras}
	\address[Athanasios E. Tzavaras]{Computer, Electrical, Mathematical Sciences \& Engineering Division, King Abdullah University of Science and Technology (KAUST), Thuwal, Saudi Arabia. }

	\email{athanasios.tzavaras@kaust.edu.sa}
	
	
	\begin{abstract}
	We embed the equations of polyconvex thermoviscoelasticity into an augmented, symmetrizable,  hyperbolic system and derive a relative entropy identity in the extended variables. Following the relative entropy formulation, we prove the convergence from thermoviscoelasticity with  Newtonian viscosity and Fourier heat conduction to 
	smooth solutions of the system of adiabatic thermoelasticity  as both parameters tend to zero. Also, convergence from thermoviscoelasticity to 
	smooth solutions of thermoelasticity in the zero-viscosity limit. Finally, we establish a weak-strong uniqueness result for the equations of adiabatic thermoelasticity 
	in the class of entropy weak solutions.
	\end{abstract}
	
	\maketitle
	
\baselineskip=14pt
	
\section{Introduction}
Systems of conservation laws
\begin{equation}
\label{intro-cl}
\partial_t A(U) +\partial_{\alpha}f_{\alpha}(U)=0
\end{equation}
$U : \mathbb{R}^d \times \mathbb{R}^+ \to \mathbb{R}^n$, are often equipped with an additional conservation law,
\begin{equation}
\label{intro-entropy}
\partial_t H(U) +\partial_{\alpha} q_{\alpha}(U)=0 \, .
\end{equation}
When the entropy $H$ is convex, as a function of the conserved variable $V = A(U)$, the system is called symmetrizable and \eqref{intro-cl} is hyperbolic.
The class of symmetrizable systems encompasses important examples from applications - most notably the equations of gas dynamics -
and were singled out as a class by Lax and Friedrichs \cite{MR0285799}. 
The remarkable stability properties induced by convex entropies are captured by the relative entropy
method of Dafermos \cite{doi:10.1080/01495737908962394,MR546634} and DiPerna \cite{MR523630}, as extended 
for the system \eqref{intro-cl}-\eqref{intro-entropy}
by Christoforou-Tzavaras \cite{christoforou2016relative}.
However, many  thermomechanical systems do not fit under the framework of symmetrizable systems. The reason is
that convexity of the entropy is too restrictive a condition, due to the requirement that thermomechanical systems 
need to comply with the principle of frame-indifference, see \cite{TN1992}. The objective of this article is to discuss 
this issue of stability in situations where a system is generated by a polyconvex free energy.

To put this issue into perspective, consider the system of thermoviscoelasticity,
\begin{align}
\label{sys1}
\partial_t F_{i\alpha}-\partial_{\alpha}v_i&=0 \nonumber\\ 
\partial_t v_i-\partial_{\alpha}\Sigma_{i\alpha}&=\partial_{\alpha}Z_{i\alpha}\\
\partial_t\left(\frac{1}{2}|v|^2+e\right)-\partial_{\alpha}(\Sigma_{i\alpha}v_i)&=\partial_{\alpha}(Z_{i\alpha}v_i)+\partial_{\alpha}Q_{\alpha}+r \, ,  \nonumber
\end{align}
which describes the evolution of a thermomechanical process $\big ( y(x,t) , \theta(x,t) \big) \in \mathbb{R}^3\times\mathbb{R}^+$ with $(x,t)\in\mathbb{R}^3\times\mathbb{R}^+$.
Note that $F \in \mathbb{M}^{3\times3}$ stands for the deformation gradient,  $F = \nabla y$, while $v = \partial_t y$ is the velocity
and $\theta$ is the temperature. It is written in  (\ref{sys1}) as a system of first order equations, with the
first equation in (\ref{sys1})  describing compatibility among the referential velocity gradient
and the time derivative of the deformation gradient. One also  needs to append the constraint 
\begin{equation}
\label{constraint}
\partial_{\alpha}F_{i\beta}=\partial_{\beta}F_{i\alpha}, \qquad i,\alpha,\beta=1,2,3 \, ,
\end{equation}
securing that $F$ is a gradient. The constraint \eqref{constraint} is an involution inherited from the initial data via \eqref{sys1}$_1$.
The remaining variables in (\ref{sys1}) stand for the total stress $\Sigma_{i \alpha } + Z_{i \alpha}$, the internal energy $e$, the
referential heat flux $Q_\alpha$ and the radiative heat supply $r$. The second and third equations in (\ref{sys1})
describe the balance of linear momentum and the balance of energy, respectively. For simplicity we have normalized the reference density $\rho_0 =1$.

The system is closed through constitutive relations. The names of thermomechanical theories reflect the prime variables selected to
describe the respective constitutive theories. For the constitutive theory of thermoviscoelasticity, the prime variables are 
the deformation gradient $F$, velocity gradient $\dot F$,  the temperature $\theta$ and its gradient $G = \nabla \theta$. 
Constitutive theories are required to be consistent with the Clausius-Duhem inequality for smooth thermomechanical processes, what imposes
several restrictions in their format, see \cite{CN1963,TN1992}. In the case of thermoviscoelasticity, the requirement of consistency dictates 
that the elastic stresses $\Sigma$, the entropy $\eta$ and the internal energy $e$  are generated
from the free-energy function $\psi$ via the constitutive theory
\begin{align}
\label{con.rel.aug}
\psi=\psi(F,\theta),\quad
\Sigma=\frac{\partial\psi}{\partial F},\quad
\eta=-\frac{\partial\psi}{\partial\theta},\quad
e=\psi+\theta\eta \, .
\end{align}
Moreover, the total stress is decomposed into an elastic part $\Sigma$ (as above) and a viscoelastic part $Z = Z(F, \dot{F}, \theta, G)$. The viscoelastic
stresses $Z_{i \alpha}$ with $i, \alpha = 1,2,3$, and the referential heat flux  vector $Q = Q(F, \dot{F}, \theta, G)$, with components $Q_\alpha$, $\alpha =1,2,3$, 
must satisfy
\begin{equation}
\label{con.ineq}
\tfrac{1}{\theta^2} Q(F, \dot{F}, \theta, G) \cdot G + \tfrac{1}{\theta}  Z (F, \dot{F}, \theta, G)  : \dot F \ge 0 \qquad \forall  (F, \dot{F}, \theta, G) \, .
\end{equation}
 As a consequence of  thermodynamic consistency with the Clausius-Duhem inequality and its implications \eqref{con.rel.aug}-\eqref{con.ineq},
smooth processes satisfy the entropy production identity
\begin{equation} 
\label{entr.prod}
\begin{aligned}
\partial_t\eta-\partial_{\alpha}\left(\frac{Q_{\alpha}}{\theta}\right) 
=  \nabla\theta\cdot\frac{Q}{\theta^2} +  \nabla v :  \frac{1}{\theta} Z +\frac{r}{\theta}
\ge \frac{r}{\theta}.
 \end{aligned}
\end{equation}
There are two further related thermomechanical theories: the theory of thermoelasticity with prime variables in the constitutive relations $F$, 
$\theta$, $G$ and the theory of adiabatic thermoelasticity with prime variabled $F$, $\theta$ describing a thermoelastic nonconductor of heat.
From a perspective of thermodynamical structure, these emerge from \eqref{con.rel.aug}-\eqref{con.ineq} as the natural special cases, see \cite{CN1963,TN1992}.

A natural question is to consider the convergence from thermoviscoelasticity (\ref{sys1})-(\ref{entr.prod}) 
to either (i) adiabatic thermoelasticity in the limit as thermal conductivity and viscosity tend to zero,  or (ii) thermoelasticity
when only the viscosity tends to zero but thermal conductivity is kept constant.
When the free energy satisfies 
\begin{equation}
\label{cond.gibbs}
\psi_{FF} (F, \theta) > 0 \, , \quad \eta_\theta (F,\theta)  = - \frac{\del \psi}{\del \theta} (F, \theta) > 0 \, ,
\end{equation}
these convergences can be established by means of the relative entropy method, see \cite{MR546634,doi:10.1080/01495737908962394,christoforou2016relative}.

The conditions \eqref{cond.gibbs} coincide with the Gibbs thermodynamic stability conditions and are very natural for gas dynamics.
However, the first condition in \eqref{cond.gibbs} is too restrictive for more general thermoelastic materials. For instance, together with
the principle of frame indifference, it would exclude that free energy becomes infinite as $\det F \to 0$, a requirement necessary
to avoid interpenetration of matter. 
Two approaches are followed in the literature to extend stability properties to a more meaningful class of constitutive laws,
both applicable in an isothermal context.
One approach in Dafermos \cite{Dafermos1986}  exploits the involutive structure of the elasticity system and extends the stability properties by
using ideas from compensated compactness \cite{Tartar1979}. An alternative approach exploits the null-Lagrangian structure for the class 
of polyconvex elasticity  (of Ball \cite{MR1919825})
and embeds the system into an augmented symmetric hyperbolic system \cite{MR1651340,MR1831179,MR3468916,Wagner2009}. This approach has also been applied to
systems from electromagnetism \cite{Serre2004,MR2048567}, to approximations of isothermal polyconvex elasticity by viscosity \cite{LT2006} or relaxation
\cite{MR3232713}, and to systems describing extremal surfaces in Minkwoski space \cite{Duan2017}.
We refer to Dafermos \cite[Ch V]{MR3468916} for an outline of a general framework of such results.

Here, in the spirit of \cite{MR1919825}, we replace convexity of the free energy  by polyconvexity,
\begin{equation}
\label{cond.poly}
\psi (F, \theta) =  \hat{\psi}(\Phi(F),\theta) \, , \quad 
\mbox{where $\Phi(F)=(F,\mathrm{cof} F,\det F)\in\mathbb{M}^{3\times3}\times\mathbb{M}^{3\times3}\times\mathbb{R}$},
\end{equation}
where  $\hat{\psi}=\hat{\psi}(\xi,\theta)$ a function on $\mathbb{R}^{19}\times \mathbb{R}^+$ is assumed to satisfy
\begin{equation}
\label{cond.redgibbs}
\hat\psi_{\xi \xi} (\xi ,  \theta) > 0 \, , \quad \hat\psi_{\theta \theta} (\xi, \theta) < 0 \, .
\end{equation}
and we discuss the stability of systems in thermomechanics when they are generated by polyconvex free energies.
Following \cite{MR1651340,MR1831179}, we extend the system into an augmented symmetrizable system (see Section \ref{sec2}),
and then apply to the resulting system the relative entropy formulation following \cite{christoforou2016relative}.
This allows to carry out the convergence results from thermoviscoelasticity to adiabatic thermoelasticity, or from
thermoviscoelasticity to thermoelasticity under the framework of hypetheses \eqref{cond.poly}-\eqref{cond.redgibbs}. 
The other main result is the weak-strong uniqueness of the system 
of adiabatic thermoelasticity \eqref{sys1adiab} in the class of entropy weak solutions. 


The article is organized as follows: In Section \ref{sec2}, we embed (\ref{sys1})-(\ref{entr.prod}) into an augmented
system which is symmetrizable in the sense of \cite{MR0285799}. To motivate this, we first perform an embedding 
of the system of adiabatic thermoelasticity  (\ref{sys1adiab}) to an augmented symmetrizable, hyperbolic system. 
The extension is based on transport identities of the null-Lagrangians \cite{MR1651340,MR1831179}, 
it preserves the differential constraints \eqref{constraint} and in particular (\ref{sys1})-(\ref{con.rel.aug}) can be regarded 
as a constrained evolution of the augmented system.
In Section \ref{sec3}, we calculate the relative entropy identity (\ref{rel.en.id}) for the augmented system,
which we use in Section \ref{sec4} to prove convergence of solutions of the augmented thermoviscoelasticity 
with Newtonian viscosity and Fourier heat conduction to smooth solutions of adiabatic thermoelasticity in the limit as the parameters
$\mu, k$ tend to zero. In order to do so, we place some appropriate growth conditions on the constitutive functions 
and use the estimates in Appendix \ref{AppA},
that allow us to interpret the relative entropy as a ``metric" measuring the distance between the two solutions. Analogous results can be found
in \cite{MR2909912, MR3381190}  in the context of gas dynamics, and in
\cite{christoforou2016relative}, including bounds for the relative entropy and the relative stresses for general systems of conservation laws in 
the parabolic and hyperbolic case. In Section \ref{sec5}, we treat in a similar manner the problem of convergence from the system of augmented 
thermoviscoelasticity to thermoelasticity in the zero-viscosity limit. Section \ref{sec6} is dedicated to the study of uniqueness of
strong solutions for thermoelasticity within the class of entropy weak solutions. Finally, section \ref{sec7} uses the analysis of section \ref{sec6}
in order to prove convergence from entropy weak solutions of thermoelasticity to strong solutions of the adiabatic thermoelasticity system,
so long as the latter system admits a smooth solution. 
We refer to \cite{Dafermos1986,LT2006} for weak-strong uniqueness for isothermal elasticity, and to \cite{MR2909912} for weak-strong uniqueness
for the full Navier-Stokes-Fourier system. In a class of measure-valued solutions, weak-strong uniqueness results are available for the incompressible Euler equations  \cite{MR2805464},   for  polyconvex elastodynamics \cite{MR1831175}, and for the isothermal gas dynamics system \cite{GSW2015}.


\section{The symmetrizable extension of adiabatic polyconvex thermoelasticity} \label{sec2}

We review the system of adiabatic thermoelasticity describing a thermoelastic nonconductor of heat. The system reads
\begin{align}
\begin{split}
\label{sys1adiab}
\partial_t F_{i\alpha}-\partial_{\alpha}v_i&=0 \\ 
\partial_t v_i-\partial_{\alpha}\Sigma_{i\alpha}&=0\\
\partial_t\left(\frac{1}{2}|v|^2+e\right)-\partial_{\alpha}(\Sigma_{i\alpha}v_i)&=r \end{split}
\end{align}
where $\Sigma$, $e$ are determined by the  constitutive theory (\ref{con.rel.aug}) and $r$ is a given function. Here, $x\in\mathbb{R}^3$ stands for the Lagrangian variable, 
$t\in\mathbb{R}^+$ is time, 
$\partial_t$ is the material derivative, and  the triplet 
$(F,v,\theta)$ is a function of $(x,t)$ taking values in $\mathbb{M}^{3\times3}\times\mathbb{R}^3\times\mathbb{R}^+$. Under the constitutive theory (\ref{con.rel.aug}) 
smooth solutions of~\eqref{sys1adiab} satisfy the entropy production identity
\begin{equation}
\label{S2: entropyadiab}
\partial_t \eta =\frac{r}{\theta} \, ,
\end{equation}
in accordance with the requirement of consistency of the constitutive theory with the Clausius-Duhem inequality (see \cite[Sect. 96]{TN1992}).
It should be noted  that the entropy identity 
\eqref{S2: entropyadiab} is only valid for strong solutions. For weak solutions it is replaced by the entropy inequality
$\del_t \eta \ge \frac{r}{\theta}$  which serves as an admissibility criterion. 
We always consider $F$ to be a deformation gradient, by imposing the constraint
\begin{equation}
\label{def.grad}
\partial_{\alpha}F_{i\beta}=\partial_{\beta}F_{i\alpha}, \quad i,\alpha,\beta=1,2,3 \, ,
\end{equation}
induced as an involution from the initial data.

As the requirement of convex free energy is too restrictive and inconsistent with the principle of material frame indifference, in numerous works, 
convexity has been replaced by weaker assumptions such as polyconvexity, quasiconvexity or rank-1 convexity \cite{MR0475169,MR1919825}. 
Here, we work under the framework  of 
polyconvexity and assume that the free energy $\psi(F,\theta)$ factorizes as a strictly convex function of the minors of $F$ and $\theta$. More precisely,
\begin{equation*}
\psi(F,\theta)=\hat{\psi}(\Phi(F),\theta)\;,
\end{equation*}
where $\hat{\psi}=\hat{\psi}(\xi,\theta)$ is a strictly convex function on $\mathbb{R}^{19}\times \mathbb{R}^+$ and
\begin{equation*}
 \Phi(F)=(F,\mathrm{cof} F,\det F)\in\mathbb{M}^{3\times3}\times\mathbb{M}^{3\times3}\times\mathbb{R}
\end{equation*}
stands for the vector of null-Lagrangians: $F$, the cofactor matrix  $\mathrm{cof} F$,  and the determinant $\det F$
defined respectively for $d=3$ by
\begin{align*}
\begin{split}
(\cof F)_{i\alpha} &=
\frac{1}{2}\epsilon_{ijk}\epsilon_{\alpha\beta\gamma}
F_{j\beta}F_{k\gamma},\\
\dt F &=\frac{1}{6}\epsilon_{ijk}\epsilon_{\alpha\beta\gamma}
F_{i\alpha}F_{j\beta}F_{k\gamma} = \frac{1}{3}(\cof F)_{i \alpha}F_{i \alpha}.
\end{split}
\end{align*}
The components  $\Phi^B(F),$ $B=1,\dots,19$ are null-Lagrangians and satisfy for any motion $y(x,t)$  the Euler-Lagrange equation
\begin{equation}
\label{null-Lan}
\partial_{\alpha}\left( \frac{\del \Phi^B}{\del F_{i\alpha}} (\nabla y ) \right)=0 \, , \quad B=1,\dots,19.
\end{equation}
Using the kinematic compatibility equation (\ref{sys1adiab})$_1$ and (\ref{null-Lan}) we write
\begin{align}
\label{nul-Lang1}
\partial_t \Phi^B(F)
=\partial_{\alpha}\left(\phv\right),
\end{align}
which holds for any deformation gradient $F$ and velocity fields $v.$ This yields the following additional conservation laws
\begin{align}
\begin{split}
\frac{\partial }{\partial t} \dt F 
&= \frac{\del}{\del {x^\alpha}} \bigl((\cof F)_{i\alpha}v_i\bigr)
\\
\frac{\partial }{\partial t} (\cof F)_{k \gamma}
&= \frac{\del}{\del {x^\alpha}}
(\epsilon_{ijk}\epsilon_{\alpha\beta\gamma}F_{j\beta}v_i),
\end{split}                    \label{geomcons}
\end{align}
first observed by T. Qin \cite{MR1651340}.

The stress tensor $\Sigma$ is now expressed
in terms of the null-Langragian vector $\Phi(F)$ in the following way
\begin{align*}
\Sigma_{i\alpha}=\frac{\partial\psi}{\partial F_{i\alpha}}(F,\theta)=\frac{\partial }{\partial F_{i\alpha}}\left(\hat{\psi}(\Phi(F),\theta)\right)=\frac{\partial \hat{\psi}}{\partial\xi^B}(\Phi(F),\theta)\ph.
\end{align*}
Additionally, since the first nine components of $\Phi(F)$ are the components of $F,$ (\ref{con.rel.aug}) implies that 
we can express the entropy $\eta$ and the internal energy $e$ with respect to the null-Lagrangian vector $\Phi(F),$ namely
\begin{align*}
\begin{split}
\eta(F,\theta)&=-\frac{\partial\psi}{\partial\theta}(F,\theta)=-\frac{\partial \hat{\psi}}{\partial\theta}(\Phi(F),\theta) = \hat{\eta}( \Phi(F) , \theta)
,\\
e(F,\theta)&=\psi(F,\theta)-\theta\frac{\partial\psi}{\partial\theta}(F,\theta)=\hat{\psi}(\Phi(F),\theta)-\theta \frac{\partial\hat{\psi}}{\partial\theta}   (\Phi(F),\theta)
= \hat{e}(\Phi(F),\theta) \, ,
\end{split}
\end{align*}
where we have defined 
\begin{equation}
\begin{aligned}
\label{e.tilde}
\hat{\eta}(\xi,\theta) &:= -\frac{\partial \hat{\psi}}{\partial\theta}(\xi,\theta), \quad 
\\
\hat{e}(\xi,\theta) &:= \hat{\psi}(\xi,\theta)-\theta \frac{\partial\hat{\psi}}{\partial\theta}   (\xi,\theta) \, .
\end{aligned}
\end{equation}
In summary, we conclude that solutions of adiabatic thermoelasticity satisfy  the constraint \eqref{def.grad} along with  
\begin{align}
\begin{split}
\label{sys2}
\partial_t \Phi^B(F)-\partial_{\alpha}\left(\phv \right)&=0\\ 
\partial_t v_i-\partial_{\alpha}\left(\frac{\partial \hat{\psi}}{\partial\xi^B}(\Phi(F),\theta)\ph\right)&=0\\
\partial_t \left(\frac{1}{2}|v|^2+\hat{e}(\Phi(F),\theta)\right)-\partial_{\alpha}\left(\frac{\partial \hat{\psi}}{\partial\xi^B}(\Phi(F),\theta)\ph v_i\right)&=r
\end{split}
\end{align}
and the entropy production idenity
\begin{equation}
\label{entropy.prod}
\partial_t \hat{\eta}(\Phi(F),\theta) =\frac{r}{\theta}.
\end{equation}

Next, we embed (\ref{sys2}), \eqref{def.grad} into the augmented system
\begin{align}
\label{aug.sys.2}
\begin{split}
\partial_t \xi^B-\partial_{\alpha}\left(\phv \right)&=0 \\ 
\partial_t v_i-\partial_{\alpha}\left(\frac{\partial \hat{\psi}}{\partial\xi^B}(\xi,\theta)\ph\right)&=0\\
\partial_t \left(\frac{1}{2}|v|^2+\hat{e}(\xi,\theta)\right)-\partial_{\alpha}\left(\frac{\partial \hat{\psi}}{\partial\xi^B}(\xi,\theta)\ph v_i\right)&=r \\
\partial_{\alpha}F_{i\beta} - \partial_{\beta}F_{i\alpha} &=0
\end{split}
\end{align}
while the entropy production \eqref{entropy.prod} is embedded into
\begin{equation}
\label{aug.sys.2.e}
\partial_t \hat{\eta}(\xi,\theta) =\frac{r}{\theta}\;.
\end{equation}
The system \eqref{aug.sys.2} is formulated via the extended vector $(\xi,v,\theta),$ 
with $\xi:=(F,\zeta,w)\in\mathbb{M}^{3\times3}\times\mathbb{M}^{3\times3}\times\mathbb{R}(\simeq\mathbb{R}^{19})$.
The defining functions $\hat \psi$, $\hat e$ and $\hat \eta$ are connected through \eqref{e.tilde}.

Next, we claim the extension enjoys the following properties: 
\begin{enumerate}
	\item Whenever $F(\cdot,t=0)$ is a deformation gradient \textit{i.e.} satisfies (\ref{def.grad}), then $F(\cdot,t)$ remains gradient for all times.
	\item Smooth solutions of this system preserve the constraints $\xi^B=\Phi^B(F),$ $B=1,\dots,19,$ 
	in the sense that if $F(\cdot,t=0)$ is a gradient and $\xi(\cdot,t=0)=\Phi(F(\cdot,t=0)),$ then $F(\cdot,t)$ remains gradient for all times and 
	$\xi(\cdot,t)=\Phi(F(\cdot,t)),$ for all times; therefore, (\ref{sys1adiab}) can be viewed as a constrained evolution of the augmented thermoelasticity system 
	(\ref{aug.sys.2}).
	\item Smooth solutions of \eqref{aug.sys.2} satisfy the additional conservation law \eqref{aug.sys.2.e}.
	\item Under the Hypothesis \eqref{cond.redgibbs}, 
	the augmented system (\ref{aug.sys.2}) is  symmetrizable in the sense of \cite{MR0285799}.
\end{enumerate}

Properties 1. and 2. are clear. To put Properties 3. and 4. into perspective, note that \eqref{aug.sys.2} is a system of
conservation laws \eqref{intro-cl} monitoring the evolution of $U \in \mathbb{R}^n$, $n=23$, where
$$
U=\left( \begin{array}{c} \xi\\ v\\ \theta  \end{array} \right ),\quad
A(U)=\left( \begin{array}{c} \xi\\ v\\ \frac{1}{2} |v|^2+\hat{e}(\xi, \theta) \end{array} \right ),\quad
f_\alpha (U)= \left( \begin{array}{c} \phv \\   \frac{\partial \hat{\psi}}{\partial\xi^B}(\xi,\theta)\ph  \\ 
           \frac{\partial \hat{\psi}}{\partial\xi^B}(\xi,\theta)\ph v_i  \end{array} \right )\; ,
$$ 
with $A(U), f_\alpha (U) : \mathbb{R}^n \to \mathbb{R}^n$ the vectors of conserved quantities and fluxes.
Next, we multiply \eqref{aug.sys.2} by  the multiplier 
\begin{equation}
\label{multiplier}
G(U)=\frac{1}{\theta} \left( \begin{array}{c} \frac{\partial\hat{\psi}}{\partial \xi} (\xi, \theta) \\ v\\ -1 \end{array} \right )\;.
\end{equation}
and use \eqref{e.tilde} and the null-Lagrangian property \eqref{null-Lan} to deduce
$$
- \del_t  \hat \eta (\xi, \theta) = - \frac{r}{\theta}
$$
which is an additional conservation law of the form \eqref{intro-entropy} with (mathematical) entropy $H(U) = - \hat \eta (\xi, \theta)$ and 
entropy flux $q_\alpha = 0$. This yields Property 3.

We place now Hypothesis \eqref{cond.redgibbs} which may be equivalently expressed as
\begin{equation}
\label{cond.redgibbs2}
\hat\psi_{\xi \xi} (\xi ,  \theta) > 0 \, , \quad \hat\eta_{\theta} (\xi, \theta) > 0 \, ,
\end{equation}
and recall that, in view of \eqref{e.tilde}, we have $\frac{\del \hat e}{\del \theta} = \theta \frac{\del \hat \eta}{\del \theta} > 0$. Hence,
$A : \mathbb{R}^{n} \to \mathbb{R}^{n}$ is one-to-one and $\nabla A(U)$ is nonsingular. 
It is shown in \cite{christoforou2016relative} that symmetrizability, \textit{i.e.} convexity of the entropy $H$ with respect to the conserved variable
$V=A(U)$, can be equivalently expressed by the positivity of the symmetric matrix
$$
 \nabla^2 H (U) - \sum_{k=1, ..., n} G^k (U) \nabla^2 A^k (U) > 0 \, .
$$
The latter is computed,
$$
\begin{aligned}
\nabla^2 H (U) - \sum_{k=1, ..., n} G^k (U) \nabla^2 A^k (U) 
&= - \nabla^2_{(\xi,v,\theta)} \hat \eta(\xi,\theta)  + \frac{1}{\theta} \nabla^2_{(\xi,v,\theta)} \big ( \tfrac{1}{2} |v|^2 + \hat e (\xi,\theta) \big )
\\
& \stackrel{\eqref{e.tilde}}{=} 
 \begin{pmatrix}
 \tfrac{1}{\theta} \hat \psi_{\xi \xi} & 0 & 0 \\ 0 & \tfrac{1}{\theta} I_{3 \times 3} & 0  \\ 
0 & 0 & \tfrac{1}{\theta} \hat \eta_{\theta}  
\end{pmatrix}
\; > 0 \, ,
\end{aligned}
$$
by \eqref{cond.redgibbs}. Hence, \eqref{aug.sys.2} is symmetrizable and thus hyperbolic.

The same extension applies to the system of thermoviscoelasticity (\ref{sys1}) - (\ref{entr.prod}), and the associated 
augmented system reads
\begin{align}
\begin{split}
\label{aug.sys}
\partial_t \xi^B-\partial_{\alpha}\left(\phv \right)&=0 \\ 
\partial_t v_i-\partial_{\alpha}\left(\frac{\partial \hat{\psi}}{\partial\xi^B}(\xi,\theta)\ph\right)&=\partial_{\alpha}Z_{i\alpha} \\
\partial_t \left(\frac{1}{2}|v|^2+\hat{e}\right)-\partial_{\alpha}\left(\frac{\partial \hat{\psi}}{\partial\xi^B}(\xi,\theta)\ph v_i\right)&=\partial_{\alpha}(Z_{i\alpha}v_i)+\partial_{\alpha}Q_{\alpha}+r \\
\partial_{\alpha}F_{i\beta} - \partial_{\beta}F_{i\alpha} &=0
\end{split}
\end{align}
while the entropy identity takes the form:
\begin{equation}\label{aug.sys.e}
\partial_t\hat{\eta}-\partial_{\alpha}\left(\frac{Q_{\alpha}}{\theta}\right)=\nabla\theta\cdot\frac{Q}{\theta^2}+\frac{1}{\theta} (\partial_{\alpha}v_i)
Z_{i\alpha}+\frac{r}{\theta}\;.
\end{equation}
The extension bears the same properties as listed in the case of adiabatic thermoelasticity.

In Sections \ref{sec4} and \ref{sec5}, we consider thermoviscoelasticity 
with Newtonian viscosity and Fourier heat conduction, respectively
\begin{equation}
\label{vcoeff}
\begin{aligned}
Z &= \mu\nabla v \quad \mu=\mu(F,\theta)>0 \, ,
\\
Q &= k\nabla\theta \quad k=k(F,\theta)>0 \, .
\end{aligned}
\end{equation}
We employ (with a slight abuse) the notation 
$$
\mu (\xi, \theta) =  \mu(F,\theta) \, , \quad  k (\xi, \theta) = k (F,\theta)
$$
for the extended viscosity and heat conduction coefficients. 
The augmented systems~\eqref{aug.sys.2} and~(\ref{aug.sys}) belong
to the general class of hyperbolic-parabolic systems of the form 
\begin{equation}
\label{general.hyper-par}
\partial_t A(U)+\partial_{\alpha}f_{\alpha}(U)=\varepsilon\partial_{\alpha}(B_{\alpha\beta}(U)\partial_{\beta}U) \, ,
\end{equation}
where $U=U(x,t)\in\mathbb{R}^n,$ $x\in\mathbb{R}^d,$ $t\in\mathbb{R}^+$ while $A,f_{\alpha}:\mathbb{R}^n\to\mathbb{R}^n$ and 
$B_{\alpha\beta}:\mathbb{R}^n\to\mathbb{R}^{n\times n}$.
A relative entropy identity for this class of systems, assuming that the hyperbolic part is symmetrizable in the sense of Friedrichs and Lax~\cite{MR0285799}, 
has been developed in~\cite{christoforou2016relative} under appropriate hypotheses. 
In short, it is required that $A(U)$ is globally invertible, there exists an entropy-entropy flux pair $(H,q)$, 
\textit{i.e.} there exists a smooth multiplier $G(U):\mathbb{R}^{n}\to\mathbb{R}^{n}$ such that
\begin{align}
\begin{split}
\nabla H&=G\cdot\nabla A\\
\nabla q_\alpha&=G\cdot\nabla f_\alpha,\qquad \alpha=1\dots,d
\end{split}
\end{align}
and additional hypotheses on dissipative structure of the parabolic part. Smooth solutions to systems~\eqref{general.hyper-par} satisfy the additional identity
\begin{align}
\begin{split}
\label{general.e-eflux}
\partial_t H(U) + \partial_{\alpha}q_{\alpha}(U)&=\varepsilon\partial_{\alpha}(G(U)\cdot B_{\alpha\beta}(U)\partial_{\beta}U)-\varepsilon\nabla G(U)\partial_{\alpha}U\cdot  B_{\alpha\beta}(U)\partial_{\beta}U \, .
\end{split}
\end{align}
The augmented system~\eqref{aug.sys}-\eqref{aug.sys.e} 
is of the form \eqref{general.hyper-par}--\eqref{general.e-eflux} with the multiplier $G(U)$ as in \eqref{multiplier}. 

There is available in  \cite[Theorem 2.2]{christoforou2016relative}, a general convergence theory for the zero viscosity limit 
in \eqref{general.hyper-par}: Namely, an entropy dissipation condition from \cite{MR3468916} is used:
\begin{align*}
\textit{there exists $\mu>0$ such that }\;\; \sum_{\alpha,\beta}\nabla G(U)\partial_{\alpha}U\cdot B_{\alpha\beta}(U)\partial_{\beta}U \geq 
\mu\sum_{\alpha}\left|\sum_{\beta}B_{\alpha\beta}(U)\partial_{\beta}U\right|^2  \, , 
\end{align*}
which allows dissipation to control the diffusion even for degenerate viscosity matrices, and is shown 
in  \cite{christoforou2016relative} to imply  convergence in the zero-viscosity limit for a fairly general setting of hyperbolic-parabolic systems.
Unfortunately, the augmented system of thermoviscoelasticity does not satisfy this condition and the convergence as the viscosity 
and heat-conductivity tend to zero will be handled as a special case.

\section{The relative entropy identity} \label{sec3}
Consider two smooth solutions of (\ref{aug.sys}), $U=(\xi,v,\theta)^T$ and $\bar{U}=(\bar{\xi},\bar{v},\bar{\theta})^T$
and set the mathematical entropy $H(U)$ to be the negative of the thermodynamic entropy, $-\hat{\eta}(\xi,\theta).$
We define the relative entropy and the corresponding relative flux
\begin{align}
\label{defrelen}
H(U|\bar{U})&=-\hat{\eta}(U)+\hat{\eta}(\bar{U})-G(\bar{U})\cdot(A(U)-A(\bar{U})) \nonumber\\
&=\frac{1}{\bar{\theta}}\left[\hat{\psi}(\xi,\theta|\bar{\xi},\bar{\theta})+\frac{1}{2}|v-\bar{v}|^2+(\hat{\eta}-\bar{\hat{\eta}})(\theta-\bar{\theta})\right]\\
\label{re}
q_{\alpha}(U|\bar{U})&=q_{\alpha}(U)-q_{\alpha}(\bar{U})-G(\bar{U})\cdot(f_{\alpha}(U)-f_{\alpha}(\bar{U})) \nonumber
\end{align}
where
\begin{equation*} 
\hat{\psi}(\xi,\theta|\bar{\xi},\bar{\theta})=\hat{\psi}-\bar{\hat{\psi}}-\pgb(\xi^B-\bar{\xi}^B)+\bar{\hat{\eta}}(\theta-\bar{\theta}),
\end{equation*}
and using that $\hat{\psi}_\theta=-\hat{\eta}$. Throughout the paper, we use for convenience the abbreviation $\hat{\eta}=\hat{\eta}(U),$ $\bar{\hat{\eta}}=\hat{\eta}(\bar{U}),$ $\hat{\psi}=\hat{\psi}(\xi,\theta),$ $\bar{\hat{\psi}}=\hat{\psi}(\bar{\xi},\bar{\theta})$ and so on.
We write the equations (\ref{aug.sys}) for each solution $U$ and $\bar{U}$ respectively, we subtract
and multiply by $-\bar{\theta}G(\bar{U})=\left(-\pgb,-\bar{v},1\right)^T.$ Rearranging some terms we obtain
\begin{small}
	\begin{align}
	\label{entr.eq.1}
	\begin{split}
	\partial_t&\left[\hat{\psi}(\xi,\theta|\bar{\xi},\bar{\theta})+\frac{1}{2}|v-\bar{v}|^2+(\hat{\eta}-\bar{\hat{\eta}})(\theta-\bar{\theta})\right]
	-\partial_t(-\bar{\theta}(\hat{\eta}-\bar{\hat{\eta}}))\\
	&\!\!\!\!+\partial_{\alpha}\!\left[\pgb\!\left(\phv-\phbv\right)\!-\!\left(\sit\right)\!\!(v_i-\bar{v}_i)\right.\\
	& \qquad\qquad\qquad\qquad\qquad\qquad\qquad\qquad\qquad\qquad\quad\quad -(Z_{i\alpha}-\bar{Z}_{i\alpha})(v_i-\bar{v}_i)-(Q_{\alpha}-\bar{Q}_{\alpha})\bigg]\\
	&=-\partial_t\left(\pgb\right)\!\!(\xi^B-\bar{\xi}^B)+\partial_{\alpha}\left(\pgb\right)\left(\phv-\phbv\right)\\
	&\;\;\;-\partial_{\alpha}\left(\sitb\right)(v_i-\bar{v}_i)+\partial_{\alpha}\bar{v}_i\left(\sit-\sitb\right)\\
	&\;\;\;-\partial_{\alpha}\bar{Z}_{i\alpha}(v_i-\bar{v}_i)+\partial_{\alpha}\bar{v}_i(Z_{i\alpha}-\bar{Z}_{i\alpha})+r-\bar{r}.
	\end{split}
	\end{align}
\end{small}
Next, we write the difference between equations (\ref{aug.sys.e}) for the two solutions and multiply by $-\bar{\theta}:$ 
\begin{align*} 
\partial_t(-\bar{\theta}(\hat{\eta}-\bar{\hat{\eta}}))+\partial_{\alpha}\left(\bar{\theta}\frac{Q_{\alpha}}{\theta}-\bar{\theta}\frac{\bar{Q}_{\alpha}}{\bar{\theta}}\right)
=-\partial_t\bar{\theta}(\hat{\eta}-\bar{\hat{\eta}})+\partial_{\alpha}\bar{\theta}\left(\frac{Q_{\alpha}}{\theta}-\frac{\bar{Q}_{\alpha}}{\bar{\theta}}\right)\\
-\bar{\theta}\left(\frac{1}{\theta}\partial_{\alpha}v_i Z_{i\alpha}-\frac{1}{\bar{\theta}}\partial_{\alpha}\bar{v}_i \bar{Z}_{i\alpha}\right)
-\bar{\theta}\left(\nabla\theta\cdot\frac{Q}{\theta^2}-\nabla\bar{\theta}\cdot\frac{\bar{Q}}{\bar{\theta}^2}+\frac{r}{\theta}-\frac{\bar{r}}{\bar{\theta}}\right).
\end{align*}
Therefore, setting
\begin{align*}
\begin{split}
I_1&=-\bar{\theta}\left(\frac{1}{\theta}\partial_{\alpha}v_i Z_{i\alpha}-\frac{1}{\bar{\theta}}\partial_{\alpha}\bar{v}_i \bar{Z}_{i\alpha}\right),\\
I_2&=\nabla\bar{\theta}\cdot\left(\frac{Q}{\theta}-\frac{\bar{Q}}{\bar{\theta}}\right)
-\bar{\theta}\left(\nabla\theta\cdot\frac{Q}{\theta^2}-\nabla\bar{\theta}\cdot\frac{\bar{Q}}{\bar{\theta}^2}\right),\\
I_3&=r-\bar{r}-\bar{\theta}\left(\frac{r}{\theta}-\frac{\bar{r}}{\bar{\theta}}\right)
\end{split}
\end{align*}
equality (\ref{entr.eq.1}) becomes
\begin{small}
\begin{align}
	\partial_t&\left[\hat{\psi}(\xi,\theta|\bar{\xi},\bar{\theta})+\frac{1}{2}|v-\bar{v}|^2+(\hat{\eta}-\bar{\hat{\eta}})(\theta-\bar{\theta})\right] \nonumber \\
	&\!\!\!\!\!+\partial_{\alpha}\!\left[\pgb\!\left(\phv-\phbv\right)\!-\!        \sit         (v_i-\bar{v}_i)\right. \nonumber  \\
	& \qquad\qquad\qquad\qquad\qquad\qquad\qquad -(Z_{i\alpha}-\bar{Z}_{i\alpha})(v_i-\bar{v}_i)-(Q_{\alpha}-\bar{Q}_{\alpha})
	\left.+\bar{\theta}\!\left(\frac{Q_{\alpha}}{\theta}-\frac{\bar{Q}_{\alpha}}{\bar{\theta}}\right)\right] \nonumber  \\
	&=-\partial_t\bar{\theta}(\hat{\eta}-\bar{\hat{\eta}})
	-\partial_t\left(\pgb\right)\!\!(\xi^B-\bar{\xi}^B)+\partial_{\alpha}\left(\pgb\right)\left(\phv-\phbv\right) \nonumber \\
	&\;\;\;-\partial_{\alpha}\left(\sitb\right)(v_i-\bar{v}_i)+\partial_{\alpha}\bar{v}_i\left(\sit-\sitb\right) \nonumber \\
	&\;\;\;-\partial_{\alpha}\bar{Z}_{i\alpha}(v_i-\bar{v}_i)+\partial_{\alpha}\bar{v}_i(Z_{i\alpha}-\bar{Z}_{i\alpha})+I_1+I_2+I_3\;.
\label{iden.interm}
\end{align}
\end{small}
Using the null-Lagrangian property~\eqref{null-Lan} and system~\eqref{aug.sys}, it yields	
\vfil\eject
\begin{small}
	\begin{align}\label{III}
	\partial_t&\left[\hat{\psi}(\xi,\theta|\bar{\xi},\bar{\theta})+\frac{1}{2}|v-\bar{v}|^2+(\hat{\eta}-\bar{\hat{\eta}})(\theta-\bar{\theta})\right]\nonumber\\
	&\!\!\!\!\!+\partial_{\alpha}\!\left[\pgb\!\left(\phv-\phbv\right)\!-\!        \sit         (v_i-\bar{v}_i)\right. \nonumber\\\
	& \qquad\qquad\qquad\qquad\qquad\qquad\qquad\qquad -(Z_{i\alpha}-\bar{Z}_{i\alpha})(v_i-\bar{v}_i)-(Q_{\alpha}-\bar{Q}_{\alpha})
	\left.+\bar{\theta}\!\left(\frac{Q_{\alpha}}{\theta}-\frac{\bar{Q}_{\alpha}}{\bar{\theta}}\right)\right]\nonumber\\
	&=
	-\partial_t\bar{\theta}\left(\hat{\eta}-\bar{\hat{\eta}}+\frac{\partial^2 \hat{\psi}}{\partial \xi^B\partial\theta}(\bar{\xi},\bar{\theta})(\xi^B-\bar{\xi}^B)
	-\frac{\partial\bar{\hat{\eta}}}{\partial\theta}(\theta-\bar{\theta})\right)
	-\partial_t\bar{\theta}  \frac{\partial\bar{\hat{\eta}}}{\partial\theta}(\theta-\bar{\theta})\nonumber\\
	&\;\;\;+\partial_t\bar{\xi}^B\left(\pg-\pgb-\frac{\partial^2 \hat{\psi}}{\partial\xi^B\partial\xi^A}(\bar{\xi},\bar{\theta})(\xi^A-\bar{\xi}^A)
	-\frac{\partial^2 \hat{\psi}}{\partial\xi^B\partial\theta}(\bar{\xi},\bar{\theta})(\theta-\bar{\theta})\!\right)\nonumber\\
	&\;\;\;-\partial_t\bar{\xi}^B\left(\pg-\pgb +\frac{\partial\bar{\hat{\eta}}}{\partial\theta} (\theta-\bar{\theta})\!\right)\!+\!\partial_{\alpha}\left(\pgb\right)\!\!\left(\phv-\phbv\right) \nonumber\\
	&\;\;\;-\partial_{\alpha}\left(\sitb\right)(v_i-\bar{v}_i)+\partial_{\alpha}\bar{v}_i\left(\sit-\sitb\right)\nonumber\\
	&\;\;\;-\partial_{\alpha}\bar{Z}_{i\alpha}(v_i-\bar{v}_i)+\partial_{\alpha}\bar{v}_i(Z_{i\alpha}-\bar{Z}_{i\alpha})+I_1+I_2+I_3\nonumber\\
	&=-\partial_t\bar{\theta}\hat{\eta}(\xi,\theta|\bar{\xi},\bar{\theta})+\partial_t\bar{\xi}^B\frac{\partial \hat{\psi}}{\partial \xi^B}(\xi,\theta|\bar{\xi},\bar{\theta})
	-\partial_t\bar{\theta}\,\frac{\partial\bar{\hat{\eta}}}{\partial\theta}(\theta-\bar{\theta})+\partial_t\bar{\xi}\,\frac{\partial^2 \hat{\psi}}{\partial\xi^B\partial\theta}(\bar{\xi},\bar{\theta})(\theta-\bar{\theta})\nonumber\\
	&\;\;\;-\phb\partial_{\alpha}\bar{v}_i\left(\pg-\pgb\right)+\partial_{\alpha}\left(\pgb\right)\left(\phv-\phbv\right)\nonumber\\
	&\;\;\;-\phb\partial_{\alpha}\left(\pgb\right)(v_i-\bar{v}_i)+\partial_{\alpha}\bar{v}_i\left(\sit-\sitb\right)\nonumber\\
	&\;\;\;-\partial_{\alpha}\bar{Z}_{i\alpha}(v_i-\bar{v}_i)+\partial_{\alpha}\bar{v}_i(Z_{i\alpha}-\bar{Z}_{i\alpha})+I_1+I_2+I_3,
	\end{align}
\end{small}
where
\begin{small}
	\begin{align}
	\label{partial.eta.rel}\hat{\eta}(\xi,\theta|\bar{\xi},\bar{\theta})&:=\hat{\eta}({\xi},{\theta})-{\hat{\eta}}(\bar{\xi},\bar{\theta})-\frac{\partial{\hat{\eta}}}{\partial \xi^B}(\bar{\xi},\bar{\theta})(\xi^B-\bar{\xi}^B)
	-\frac{\partial{\hat{\eta}}}{\partial\theta}(\bar{\xi},\bar{\theta})(\theta-\bar{\theta})\\
	\label{partial.g.rel}
	\frac{\partial \hat{\psi}}{\partial \xi^B}(\xi,\theta|\bar{\xi},\bar{\theta})&\!:=\! \!\pg-\pgb-\frac{\partial^2 \hat{\psi}}{\partial\xi^B\partial\xi^A}(\bar{\xi},\bar{\theta})(\xi^B-\bar{\xi}^B)
	-\frac{\partial^2 \hat{\psi}}{\partial\xi^B\partial\theta}(\bar{\xi},\bar{\theta})(\theta-\bar{\theta})\,.
	\end{align}
\end{small}
Next we observe that, using the entropy identity (\ref{aug.sys.e}), we have
\begin{small}
	\begin{align*}
	-\partial_t\bar{\theta}\frac{\partial{\hat{\eta}}}{\partial\theta}(\bar\xi,\bar\theta)\,\,(\theta-\bar{\theta})
	&+\partial_t\bar{\xi}\frac{\partial^2 \hat{\psi}}{\partial\xi^B\partial\theta}(\bar{\xi},\bar{\theta})(\theta-\bar{\theta}) =-\partial_t{\hat{\eta}}(\bar\xi,\bar\theta)\,(\theta-\bar{\theta})\\
	&=-\partial_{\alpha}\left(\frac{\bar{Q}_{\alpha}}{\bar{\theta}}(\theta-\bar{\theta})\right)+
	\frac{\bar{Q}_{\alpha}}{\bar{\theta}}\partial_{\alpha}(\theta-\bar{\theta})-
	\left(\nabla\bar{\theta}\cdot\frac{\bar{Q}}{\bar{\theta}^2}+\frac{1}{\bar{\theta}}\partial_{\alpha}\bar{v}_i \bar{Z}_{i\alpha}+\frac{\bar{r}}{\bar{\theta}}\right)(\theta-\bar{\theta}).
	\end{align*}
\end{small}
Using the null-Lagrangian property~\eqref{null-Lan}, we rearrange the remaining terms as follows
\begin{small}
	\begin{align*}
	&-\phb\partial_{\alpha}\bar{v}_i\!\left(\pg-\pgb\!\right)+\partial_{\alpha}\left(\pgb\right)\!\left(\phv-\phbv\right)\\
	&-\phb\partial_{\alpha}\left(\pgb\right)(v_i-\bar{v}_i)\!+\!\partial_{\alpha}\bar{v}_i\!\left(\sit-\sitb\right)\\
	&\quad\quad=\partial_{\alpha}\left(\pgb\right)\!\left(\ph-\phb\right)\!(v_i-\bar{v}_i)\\
	&\quad\quad\;\;+\partial_{\alpha}\bar{v}_i\left(\pg-\pgb\right)\,\left(\ph-\phb\right)\\
	&\quad\quad\;\;+\partial_{\alpha}\left(\pgb\left(\ph-\phb\right)\bar{v}_i\right)\;.
	\end{align*}
\end{small}
Substituting the above relations in (\ref{III}), we give to the relative entropy identity its final form
\begin{footnotesize}
	\begin{align}
	\label{rel.en.id}
	\begin{split}
	&\partial_t\left[\hat{\psi}(\xi,\theta|\bar{\xi},\bar{\theta})+\frac{1}{2}|v-\bar{v}|^2+(\hat{\eta}(\xi,\theta)-{\hat{\eta}}(\bar\xi,\bar\theta))(\theta-\bar{\theta})\right]\\
	&-\partial_{\alpha}\left[\left(\pg-\pgb\right)\ph(v_i-\bar{v}_i)+(Z_{i\alpha}-\bar{Z}_{i\alpha})(v_i-\bar{v}_i)
	+(\theta-\bar{\theta})\left(\frac{Q_{\alpha}}{\theta}-\frac{\bar{Q}_{\alpha}}{\bar{\theta}}\right)\right]\\
	&=-\partial_t\bar{\theta}\hat{\eta}(\xi,\theta|\bar{\xi},\bar{\theta})+\partial_t\bar{\xi}^B\frac{\partial \hat{\psi}}{\partial \xi^B}(\xi,\theta|\bar{\xi},\bar{\theta})
	+\partial_{\alpha}\left(\pgb\right)\left(\ph-\phb\right)(v_i-\bar{v}_i)\\
	&\;\;+\partial_{\alpha}\bar{v}_i  \left(\pg-\pgb\right)\left(\ph-\phb\right)\\
	&\;\;-\theta\bar{\theta}\left(\frac{\partial_{\alpha}v_i}{\theta}-\frac{\partial_{\alpha}\bar{v}_i}{\bar{\theta}}\right)
	\left(\frac{Z_{i\alpha}}{\theta}-\frac{\bar{Z}_{i\alpha}}{\bar{\theta}}\right)
	-\left(\bar\theta\dfrac{Q}{\theta}-\theta\dfrac{\bar Q}{\bar\theta}\right) \cdot
	\left(\frac{\nabla\theta}{\theta}-\frac{\nabla\bar{\theta}}{\bar{\theta}}\right)
	+(\theta-\bar{\theta})\left(\frac{r}{\theta}-\frac{\bar{r}}{\bar{\theta}}\right).
	\end{split}
	\end{align}
\end{footnotesize}


\section{Convergence from thermoviscoelasticity to adiabatic thermoelasticity} \label{sec4}
The aim of this section is to prove the convergence of smooth solutions of the augmented system of thermoviscoelasticity~\eqref{aug.sys} to a smooth solution of the augmented system of adiabatic thermoelasticity~\eqref{aug.sys.2}, 
as the appropriate parameters vanish. More precisely, consider a smooth solution $U=(\xi,v,\theta)^T$ of the augmented system (\ref{aug.sys}), with the constitutive relations (\ref{con.rel.aug}), for a Newtonian viscous fluid $Z=\mu\nabla v,$
$\mu=\mu(F,\theta)>0,$ Fourier heat conduction $Q=k\nabla\theta,$  $k=k(F,\theta)>0$ and $r=0.$ Also, let $\bar{U}=(\bar{\xi},\bar{v},\bar{\theta})^T$ be a smooth solution of (\ref{aug.sys.2}) and (\ref{con.rel.aug}) with $\bar{r}=0$. The goal is to show that $U=(\xi,v,\theta)^T$ converges to $\bar{U}=(\bar{\xi},\bar{v},\bar{\theta})^T$ as $\mu, k\to 0^+$ and we will accomplish this by ``measuring the distance" between the two solutions via the relative entropy.

Under these assumptions, the relative entropy identity 
(\ref{rel.en.id}) can be written as 
	\begin{align}
	\label{rel.en.id.N.F}
	\partial_t&\left[\hat{\psi}(\xi,\theta|\bar{\xi},\bar{\theta})+\frac{1}{2}|v-\bar{v}|^2+(\hat{\eta}-\bar{\hat{\eta}})(\theta-\bar{\theta})\right]\nonumber\\
	&-\partial_{\alpha}\!\left[\left(\pg\!-\!\pgb\right)\!\!\ph(v_i-\bar{v}_i)+\mu\partial_{\alpha}v_i(v_i-\bar{v}_i)
	+(\theta-\bar{\theta})k\frac{\nabla\theta}{\theta}\right]\nonumber\\
	=&-\partial_t\bar{\theta}\hat{\eta}(\xi,\theta|\bar{\xi},\bar{\theta})+\partial_t\bar{\xi^{B}}\frac{\partial \hat{\psi}}{\partial \xi^{B}}(\xi,\theta|\bar{\xi},\bar{\theta})\nonumber\\
	&+\partial_{\alpha}\left(\pgb\right)\left(\ph-\phb\right)(v_i-\bar{v}_i)\nonumber\\
	&+\partial_{\alpha}\bar{v}_i \left(\pg-\pgb\right) \left(\ph-\phb\right)\nonumber\\
	&-\bar{\theta}\mu\frac{|\nabla v|^2}{\theta}-\bar{\theta}k\frac{|\nabla\theta|^2}{\theta^2}+\mu\nabla v\cdot\nabla\bar{v}+k\frac{\nabla\theta\cdot\nabla\bar{\theta}}{\theta}.
	\end{align}

Before we proceed, let us impose the growth conditions on the extended variables $\hat{e}$ and $\hat{\psi}$. Upon setting
\begin{equation*}
|\xi|_{p,q,r}:=|F|^p+|\zeta|^q+|w|^r\;,
\end{equation*}
for the vector $\xi=(F,\zeta,w)$, we postulate the growth conditions
\begin{equation}
\label{gr.con.1}
c(|\xi|_{p,q,r}+\theta^{\ell})-c\leq\hat{e}(\xi,\theta)\leq c(|\xi|_{p,q,r}+\theta^{\ell})+c\;,
\end{equation}
\begin{equation}
\label{gr.con.2}
\lim_{|\xi|_{p,q,r}+\theta^{\ell}\to\infty}\frac{|\partial_F\hat{\psi}|+|\partial_{\zeta}\hat{\psi}|^{\frac{p}{p-1}}+|\partial_w\hat{\psi}|^{\frac{p}{p-2}}}{|\xi|_{p,q,r}+\theta^{\ell}}=0\;,
\end{equation}
and
\begin{equation}
\label{gr.con3}
\lim_{|\xi|_{p,q,r}+\theta^{\ell}\to\infty}\frac{|\partial_{\theta}\hat{\psi}|}{|\xi|_{p,q,r}+\theta^{\ell}}=0
\end{equation}
for some constant $c>0$ and $p\ge4,$ $q,r,\ell>1.$ Under theses assumptions, we establish some useful bounds on the relative entropy in Appendix~\ref{AppA} and these bounds are employed in the following theorems.
Moreover, we define the compact set
\begin{equation}
\Gamma_{M,\delta}:=\left\{(\bar{\xi},\bar{v},\bar{\theta}):: |\bar{F}|\leq M, |\bar{\zeta}|\leq M, |\bar{w}|\leq M, |\bar{v}|\leq M, \;\; 0<\delta\leq\bar{\theta}\leq M\right\}
\label{defGamma}
\end{equation}
 and also from now on we denote by $I(\xi,v,\theta|\bar{\xi},\bar{v},\bar{\theta})$ the quantity $\bar{\theta} H(U|\bar{U}),$ namely 
\begin{equation}
\label{I}
I(\xi,v,\theta|\bar{\xi},\bar{v},\bar{\theta})=\hat{\psi}(\xi,\theta|\bar{\xi},\bar{\theta})+\frac{1}{2}|v-\bar{v}|^2+(\hat{\eta}-\bar{\hat{\eta}})(\theta-\bar{\theta}).
\end{equation}
Now we are ready to prove the theorem of convergence, recovering a smooth solution $\bar{U}=(\bar{\xi},\bar{v},\bar{\theta})^T$ of  (\ref{aug.sys.2}), (\ref{con.rel.aug}), in the
adiabatic limit as $\mu,k\to 0^+$. 
Let us denote the solution to (\ref{aug.sys}) (with $r=0$), (\ref{con.rel.aug}),  \eqref{vcoeff}  by  $U_{\mu,k}=(\xi,v,\theta)^T$,
 to emphasize the dependence on the parameters $\mu$ and $k.$
To set up the problem precisely, we consider the viscosity and thermal diffusivity as depending on
parameters $\mu_0$, $k_0$ measuring their respective  amplitudes
$$
\mu = \mu (F, \theta ; \mu_0 )  \, , \quad k = k(F, \theta ; k_0 )
$$
and tending to zero, $\mu \to 0$ and $k \to 0$, in the limits $\mu_0 \to 0$ and $k_0 \to 0$, respectively.
We work in a periodic domain in space $Q_T=\mathbb{T}^d\times[0,T),$ for $T\in[0,\infty)$.
Throughout we consider the case $d=3$ (the analysis also covers the case $d=2$ with less integrability 
requirements $p \ge 2$).
The theorem reads as follows.
\begin{theorem} \label{thmconv}
	\label{Ad.Th-el}
	Let $U_{\mu,k}$ be a strong solution of the system
	\begin{align}
	\begin{split}
	\label{aug.sys.mu.k}
	\partial_t \xi^B-\partial_{\alpha}\left(\phv \right)&=0\\ 
	\partial_t v_i-\partial_{\alpha}\left(\pg\ph\right)&=\partial_{\alpha}(\mu\partial_{\alpha}v_i)\\
	\partial_t \left(\frac{1}{2}|v|^2+\hat{e}\right)-\partial_{\alpha}
	\left(\pg\ph v_i\right)&=\partial_{\alpha}(\mu v_i\partial_{\alpha}v_i)+\partial_{\alpha}(k\partial_{\alpha}\theta)
	\end{split}
	\end{align}
satisfying
	\begin{equation}
	\partial_t\hat{\eta}-\partial_{\alpha}\left(k\frac{\nabla\theta}{\theta}\right)=k\frac{|\nabla\theta|^2}{\theta^2}+\mu\frac{|\nabla v|^2}{\theta}
	\end{equation}
	accompanied by the constitutive relations (\ref{con.rel.aug}) with initial data $U_{\mu,k}^0$ and defined on a maximal domain $Q_{T^*}$. Let $\bar{U}=(\bar\xi,\bar v,\bar\theta)^T$ be a smooth solution of the system of adiabatic thermoelasticity (\ref{aug.sys.2})-(\ref{aug.sys.2.e}), (\ref{con.rel.aug})
	with $\bar{r}=0$ and initial data $\bar{U}^0$ defined on $\bar{Q}_T$, $0<T<T^*$.  Suppose that $\nabla_{\xi}^2\hat{\psi}(\xi,\theta)>0$ and $\hat{\eta}_{\theta}(\xi,\theta)>0$ and the growth
	conditions (\ref{gr.con.1}), (\ref{gr.con.2}), (\ref{gr.con3}) hold true. If $\bar{U}\in\Gamma_{M,\delta},$ for some $M,\delta>0,$ then there exists a constant $C=C(T)$ such that for $t\in[0,T]$
\begin{small}
\begin{align}
		\label{res.1} 
		\int I(U_{\mu,k}(t)|\bar{U}(t))\:dx\leq C\left(\int I(U_{\mu,k}^0|\bar{U}^0)\:dx+\int_{0}^{T}\!\!\!\!\int\!\!\left(\mu\,\theta(s)\frac{|\nabla\bar{v}(s)|^2}{\bar{\theta}(s)}+k\frac{|\nabla\bar{\theta}(s)|^2}{\bar{\theta}(s)}\right)\!dxdt\right) \, .
		\end{align}
\end{small}
Moreover, if
\begin{align}
	\label{mu.k}
	|\mu(\xi,\theta)  \theta|\leq  C \mu_0|\hat{e}(\xi,\theta)| \quad\text{and}\quad |k(\xi,\theta)|\leq C k_0|\hat{e}(\xi,\theta)|,
	\end{align}
	for some positive constants $\mu_0,\,k_0$ then whenever the data satisfy
	\begin{align}
	\label{adiab.data}
	\lim_{k_0,\mu_0\to 0}\int I(U_{\mu,k}^0|\bar{U}^0)\:dx=0\;,
	\end{align}
	we have
	\begin{align}
	\label{res.2}
	\sup_{t\in[0,T)}\int I(U_{\mu,k}(t)|\bar{U}(t))\:dx\to 0, \quad\text{as $k_0,\mu_0\to 0^+.$}
	\end{align}
\end{theorem}

\begin{proof} Integrating the relative entropy identity (\ref{rel.en.id.N.F}) and using the estimates (\ref{bound5}), (\ref{bound6}), (\ref{bound7}) and (\ref{bound8}) we have
\begin{small}
	\begin{align*}
	&\frac{d}{dt}\int I(\xi,v,\theta|\bar{\xi},\bar{v},\bar{\theta})\:dx+\int \bar{\theta}\left(\mu\frac{|\nabla v|^2}{\theta}+k\frac{|\nabla\theta|^2}{\theta^2}\right)\:dx\le\\
	&\:\leq \int |\partial_t\bar{\theta}| |\hat{\eta}(\xi,\theta|\bar{\xi},\bar{\theta})|
	+|\partial_t\bar{\xi}^B| \Big|\frac{\partial \hat{\psi}}{\partial \xi^B}(\xi,\theta|\bar{\xi},\bar{\theta})\Big|
	+\Big| \partial_{\alpha}\big( \pgb\big)\Big| \Big|\big(\ph-\phb\big)(v_i-\bar{v}_i)\Big|\\
	&\qquad\qquad\qquad\qquad\:\:+|\partial_{\alpha}\bar{v}_i|   \left|    \left(\pg-\pgb\right) \left(\ph-\phb\right)\right|\:dx\\
	&\:\:\:+\left(\int\bar{\theta} \mu\frac{|\nabla v|^2}{\theta}\:dx\right)^{1/2}\left(\int \mu\,\theta\frac{|\nabla\bar{v}|^2}{\bar{\theta}}\:dx\right)^{1/2}
	+\left(\int\bar{\theta} k\frac{|\nabla\theta|^2}{\theta^2}\:dx\right)^{1/2}\left(\int k\frac{|\nabla\bar{\theta}|^2}{\bar{\theta}}\:dx\right)^{1/2}\\
	&\:\leq C \int I(\xi,v,\theta|\bar{\xi},\bar{v},\bar{\theta})\:dx+\frac{1}{2}\int\bar{\theta}\left(\mu\frac{|\nabla v|^2}{\theta}+k\frac{|\nabla\theta|^2}{\theta^2}\right)\:dx
	+\frac{1}{2}\int\left(\mu\,\theta\frac{|\nabla\bar{v}|^2}{\bar{\theta}}+k \frac{|\nabla\bar{\theta}|^2}{\bar{\theta}}\right)\:dx,
	\end{align*}
\end{small}
for some constant $C=C\left(|\partial_t\bar{\theta}|,|\partial_t\bar{\xi}|,\left|\partial_{\alpha}\pgb\right|,|\partial_{\alpha}\bar{v}_i|\right).$ Therefore it holds
\begin{small}
	\begin{align*}
		\frac{d}{dt}\int I(\xi,v,\theta|\bar{\xi},\bar{v},\bar{\theta})\:dx \leq
		C \int I(\xi,v,\theta|\bar{\xi},\bar{v},\bar{\theta})\:dx
		+\frac{1}{2}\int\left(\mu\frac{\theta|\nabla\bar{v}|^2}{\bar{\theta}}+k \frac{|\nabla\bar{\theta}|^2}{\bar{\theta}}\right)\:dx.
	\end{align*}
\end{small}
Then, Gronwall's inequality implies
\begin{small}
	\begin{align}
	\begin{split}
	\label{Gronwal.adiab.}
	\int I(\xi (t) ,v(t),\theta(t)|\bar{\xi}(t),\bar{v}(t),\bar{\theta}(t))\:dx\leq &\, e^{Ct}\int I(\xi^0 ,v^0,\theta^0|\bar{\xi}^0,\bar{v}^0,\bar{\theta}^0)\,dx\\
	&+\frac{1}{2}\int_{0}^{t} e^{C(t-s)}\int\left(\mu\,\theta(s)\frac{|\nabla\bar{v}(s)|^2}{\bar{\theta}(s)}+k\frac{|\nabla\bar{\theta}(s)|^2}{\bar{\theta}(s)}\right)\:dsdx
	\end{split}
	\end{align}
\end{small}
and (\ref{res.1}) follows immediately. 

Finally to show (\ref{res.2}), we first observe that integrating (\ref{aug.sys.mu.k})$_3$ in $\mathbb{T}^d$ allows to obtain the uniform bound
\begin{align}
\label{energy.unif}
\int_{\mathbb{T}^d} \frac{1}{2}|v|^2+\hat{e}(\xi,\theta) \:dx \leq C,
\end{align}
which in turn, by (\ref{gr.con.1}) and assumption (\ref{mu.k}), implies the uniform bounds
\begin{align*}
\|\mu(\xi,\theta)\theta(x,t)\|_{L^1(\mathbb{T}^d)}\leq \mu_0 \mathcal{C} \quad\text{and}\quad \|k(\xi,\theta)\|_{L^1(\mathbb{T}^d)}\leq k_0 \mathcal{C},
\end{align*}
for all $t\in[0,T]$ and some $\mathcal{C}>0.$ Consequentially, the last term on the right hand-side of (\ref{res.1}) is also uniformly bounded
$$\int_{0}^{T}\int\left(\mu\,\theta(s)\frac{|\nabla\bar{v}(s)|^2}{\bar{\theta}(s)}+k\frac{|\nabla\bar{\theta}(s)|^2}{\bar{\theta}(s)}\right)\!dxdt
\leq (\mu_0+k_0){\mathcal{C}}^{\prime}$$
and given (\ref{adiab.data}), the convergence result (\ref{res.2}) follows by taking the limit $k_0,\mu_0\to 0$ in (\ref{res.1}).
\end{proof}

\begin{remark}\rm
1.  The Hypothesis \eqref{mu.k} on the viscosity and heat conductivity coefficients: 
(a) It indicates that $\mu$, $k$ tend to zero as the parameters $\mu_0, k_0 \to 0$. 
(b) The dependence on $(F, \theta)$ satisfies the uniform bounds \eqref{mu.k} in
terms of the internal energy. 
To our knowledge there is no sufficient available information regarding the behavior of such constitutive functions for
large deformations and temperatures $(F, \theta)$ to gauge how meaningful such technical hypotheses  are. 
Nevertheless, under the growth hypothesis \eqref{gr.con.1} the constant functions
$\mu = \mu_0$ and $k = k_0$ satisfy \eqref{mu.k}.

2.  Although Theorem \ref{thmconv} is stated at the level of the augmented system, 
this convergence result holds when comparing thermoviscoelasticity (\ref{sys1})-(\ref{entr.prod}) with 
adiabatic thermoelasticity (\ref{sys1adiab})-(\ref{S2: entropyadiab}).
\end{remark}

\section{Convergence from thermoviscoelasticity to thermoelasticity in the zero viscosity limit} \label{sec5}
Here, we consider again a smooth solution $U_{\mu,k}=(\xi,v,\theta)^T$ of (\ref{aug.sys}), (\ref{con.rel.aug}), with $Z=\mu\nabla v,$
$\mu=\mu(F,\theta)>0,$ $Q=k\nabla\theta,$  $k=k(F,\theta)>0$ and $r=0.$ The aim is to compare $U_{\mu,k}=(\xi,v,\theta)^T$ with a solution
$\bar{U}_k=(\bar{\xi},\bar{v},\bar{\theta})^T$ of (\ref{aug.sys})-(\ref{con.rel.aug}) with $\bar{Z}=0,$ $\bar{Q}=\bar{k}\nabla\bar{\theta},$ $\bar{k}=\bar{k}(\bar{F},\bar{\theta})>0$ and $\bar{r}=0.$
Applying these assumptions, the relative entropy identity (\ref{rel.en.id}) becomes
\begin{footnotesize}
	\begin{align}
	\label{rel.en.id.heat}
	\begin{split}
	\partial_t&\left[\hat{\psi}(\xi,\theta|\bar{\xi},\bar{\theta})+\frac{1}{2}|v-\bar{v}|^2+(\hat{\eta}-\bar{\hat{\eta}})(\theta-\bar{\theta})\right]\\
	&-\partial_{\alpha}\!\left[\left(\pg\!-\!\pgb\right)\!\!\ph(v_i-\bar{v}_i)+\mu\partial_{\alpha}v_i(v_i-\bar{v}_i)
	+(\theta-\bar{\theta})\left(k\frac{\nabla\theta}{\theta}-\bar{k}\frac{\nabla\bar{\theta}}{\bar{\theta}}\right)\right]\\
	&+\bar{\theta}\mu\frac{|\nabla v|^2}{\theta} +\bar{\theta}k\left(\frac{\nabla\bar{\theta}}{\bar{\theta}}-\frac{\nabla\theta}{\theta}\right)^2\\
	&=-\partial_t\bar{\theta}\hat{\eta}(\xi,\theta|\bar{\xi},\bar{\theta})+\partial_t\bar{\xi}^B \frac{\partial \hat{\psi}}{\partial \xi^B}(\xi,\theta|\bar{\xi},\bar{\theta})
	+\partial_{\alpha}\pgb\left(\ph-\phb\right)(v_i-\bar{v}_i)\\
	&+\partial_{\alpha}\bar{v}_i  \left(\pg-\pgb\right)  \left(\ph-\phb\right)
	+\mu\nabla v\cdot\nabla\bar{v} +\left(\frac{\nabla\bar{\theta}}{\bar{\theta}}-\frac{\nabla\theta}{\theta}\right)\frac{\nabla\bar{\theta}}{\bar{\theta}}(\bar{\theta}k-\theta\bar{k}).
	\end{split}
	\end{align}
\end{footnotesize}

We work as in the previous section, using the results of Lemmas~\ref{lemma1},~\ref{lemma2} and~\ref{lemma3}, to prove convergence of smooth solutions for the system of thermoviscoelasticity
to solutions of thermoelasticity in the limit as $\mu\to 0^+.$
\begin{theorem}
	\label{Th-el}
	Let $U_{\mu,k}$ be a strong solution of (\ref{aug.sys.mu.k}) under the constitutive relations (\ref{con.rel.aug})
	with initial data $U_{\mu,k}^0$ and let $\bar{U}_k$ be a smooth solution of (\ref{aug.sys}) subject to (\ref{con.rel.aug})
	with $\bar{Z}=0,$ $\bar{Q}=\bar{k}\nabla\bar{\theta},$ $\bar{r}=0$ and initial data $\bar{U}^0.$ Suppose that $\nabla_{\xi}^2\hat{\psi}(\xi,\theta)>0$ and $\hat{\eta}_{\theta}(\xi, \theta)>0$ and the growth
	conditions (\ref{gr.con.1})--(\ref{gr.con3}) hold true. Suppose additionally that 
	\begin{align}
	\begin{split}
	\label{mu_0.k.theta}
	0<\underline{\delta}&\leq\theta(x,t),\\
	k(\xi,\theta)+\frac{\theta^2}{k(\xi,\theta)}&\leq C\:\hat{e}(\xi,\theta),\\
	|\mu(\xi,\theta)  \theta| &\leq  C \mu_0|\hat{e}(\xi,\theta)|,
	\end{split}
	\end{align}
for some constant $C=C(k_0).$
	If $\bar{U}_k\in\Gamma_{M,\delta},$ for all $(x,t)\in Q_T$ and some $M,\delta>0$, then there exists a constant 
	$C=C(T,\underline{\delta})$ such that
	\begin{align}
	\label{res.1.th.01}
	\int I(U_{\mu,k}(t)| \bar{U}_k (t))\:dx\leq C\left(\int I(U_{\mu,k}^0|\bar{U}_k^0)\:dx+\int_{0}^{T}\!\!\!\!\int\!\!\left(\mu\,\theta(s)\frac{|\nabla\bar{v}(s)|^2}{\bar{\theta}(s)}\right)dx\,dt\right).\nonumber\\
	\end{align}
	Also, 
	whenever
	\begin{align*}
	\lim_{\mu_0\to 0}\int I(U_{\mu,k}^0|\bar{U}_k^0)\:dx=0
	\end{align*}
	we have
	\begin{align}
	\label{res.2.th}
	\sup_{t\in[0,T)}\int I(U_{\mu,k}(t)|\bar{U}_k(t))\:dx\to 0, \quad\text{as $\mu_0\to 0^+.$}
	\end{align}
\end{theorem}

\begin{proof} Following along the lines of the proof of Theorem \ref{Ad.Th-el}, we integrate \eqref{rel.en.id.heat} and use (\ref{bound5}), (\ref{bound6}), (\ref{bound7}), (\ref{bound8}), (\ref{bound10}) and (\ref{mu_0.k.theta}) to get
\begin{small}
	\begin{align*}
	&\frac{d}{dt}\int I(\xi,v,\theta|\bar{\xi},\bar{v},\bar{\theta})\:dx
	+\int \bar{\theta}\left(\mu\frac{|\nabla v|^2}{\theta}+k\left(\frac{\nabla\bar{\theta}}{\bar{\theta}}-\frac{\nabla\theta}{\theta}\right)^2\right)\:dx\\
	&\:\leq \int |\partial_t\bar{\theta}| |\hat{\eta}(\xi,\theta|\bar{\xi},\bar{\theta})|+|\partial_t\bar{\xi}^B| \left|\frac{\partial \hat{\psi}}{\partial \xi^B}(\xi,\theta|\bar{\xi},\bar{\theta})\right|
	+\left|\partial_{\alpha}\pgb\right| \left|\left(\ph-\phb\right)(v_i-\bar{v}_i)\right|\\
	&\:+|\partial_{\alpha}\bar{v}_i| \left| \left(\pg-\pgb\right)  \left(\ph-\phb\right)  \right|\:dx\\
	&\:+\left(\int\bar{\theta} \mu\frac{|\nabla v|^2}{\theta}\:dx\right)^{1/2}\!\!\left(\int \mu\frac{\theta|\nabla\bar{v}|^2}{\bar{\theta}}\:dx\right)^{1/2}\!\!\!\! \\
	&+\left(\int\bar{\theta} k\left(\frac{\nabla\bar{\theta}}{\bar{\theta}}-\frac{\nabla\theta}{\theta}\right)^2\:dx\right)^{1/2}\!\!
	\left(\int\frac{|\nabla\bar{\theta}|^2}{\bar{\theta}}\frac{\theta^2}{k}\left(\frac{k}{\theta}-\frac{\bar{k}}{\bar{\theta}}\right)^2\:dx\right)^{1/2}\\
	&\:\leq C^{\prime} \int I(\xi,v,\theta|\bar{\xi},\bar{v},\bar{\theta})\:dx
	+\frac{1}{2}\int\bar{\theta}\left(\mu\frac{|\nabla v|^2}{\theta}+k\left(\frac{\nabla\bar{\theta}}{\bar{\theta}}-\frac{\nabla\theta}{\theta}\right)^2\right)\:dx\\
	&\:+\frac{1}{2}\int\left(\mu\,\theta\frac{|\nabla\bar{v}|^2}{\bar{\theta}}+\frac{|\nabla\bar{\theta}|^2}{\bar{\theta}}\frac{\theta^2}{k}\left(\frac{k}{\theta}-\frac{\bar{k}}{\bar{\theta}}\right)^2\right)\:dx\\
	&\:\leq C \int I(\xi,v,\theta|\bar{\xi},\bar{v},\bar{\theta})\:dx
	+\frac{1}{2}\int\bar{\theta}\left(\mu\frac{|\nabla v|^2}{\theta}+k\left(\frac{\nabla\bar{\theta}}{\bar{\theta}}-\frac{\nabla\theta}{\theta}\right)^2\right)\:dx
	+\frac{1}{2}\int\left(\mu\,\theta\frac{|\nabla\bar{v}|^2}{\bar{\theta}}\right)\:dx,
	\end{align*}
\end{small}for some constant $C=C\left(|\partial_t\bar{\theta}|,|\partial_t\bar{\xi}|,\left|\partial_{\alpha}\pgb\right|,|\partial_{\alpha}\bar{v}_i|,\underline{\delta},k_0\right).$
Applying Gronwall's inequality we obtain
	\begin{small}
	\begin{align*}
	\int I(\xi,v,\theta|\bar{\xi},\bar{v},\bar{\theta})\:dx\leq e^{Ct}\int\left(\hat{\psi}(\xi^0,\theta^0|\bar{\xi}^0,\bar{\theta}^0)+\frac{1}{2}|v^0-\bar{v}^0|^2+(\hat{\eta}^0-\bar{\hat{\eta}}^0)
	(\theta^0-\bar{\theta}^0)\right)\:dx\\
	+\frac{1}{2}\int_{0}^{t} e^{C(t-s)}\int\mu\,\theta(s)\frac{|\nabla\bar{v}(s)|^2}{\bar{\theta}(s)}\:dsdx
	\end{align*}
\end{small}
which gives (\ref{res.1.th.01}). The result (\ref{res.2.th}), follows from (\ref{mu_0.k.theta})$_1,$ by arguing exactly like Theorem \ref{Ad.Th-el} and by taking the limit $\mu_0\to 0$ in (\ref{res.1.th.01}).
\end{proof}

\section{Uniqueness of smooth solutions of Adiabatic Thermoelasticity in the class of entropy weak solutions} \label{sec6}

In this section we consider an entropy weak solution of the adiabatic thermoelasticity system \eqref{sys1adiab}, \eqref{def.grad},
\eqref{con.rel.aug}. The solution $(F, v,\theta)^T$ satisfies the weak form of the conservation laws \eqref{sys1adiab} and the weak
form of the entropy inequality
\begin{align}
\label{entr.in}
\partial_t \big ( \hat{\eta} (\Phi (F), \theta)  \big ) \geq\frac{r}{\theta}.
\end{align}

We employ \cite[Lemma 5]{MR1831179} which states:
\begin{lemma} \label{lemma.weak.minors}
Let $y : [0,\infty) \times \mathbb{T}^3 \to \mathbb{R}^3$ with regularity 
$y \in W^{1,\infty}(L^2(\mathbb{T}^3)) \cap L^\infty (W^{1,p} (\mathbb{T}^3))$ with $p \ge 4$,
and let $v = \partial_t y$, $F = \nabla y$. Then $(F, \cof F, \det F)$ satisfy \eqref{geomcons} in the sense of distributions.
\end{lemma}
\noindent
The regularity $p \ge 4$ is necessary in the case of $y : [0,\infty) \times \mathbb{T}^3 \to \mathbb{R}^3$. For maps 
$y : [0,\infty) \times \mathbb{T}^2 \to \mathbb{R}^2$ it is replaced by $p \ge 2$.

Using the lemma we conclude that, within the regularity framework of Lemma  \ref{lemma.weak.minors}, 
the extended function  $(\xi, v, \theta)^T$ with $\xi = \Phi (F)$ 
is a weak solution of the augmented thermoelasticity system,
\begin{align}
\label{aug.sys.ws}
&\partial_t \xi^B-\partial_{\alpha}\left(\phv \right)=0 \nonumber\\ 
&\partial_t v_i-\partial_{\alpha}\left(\sit\right)=0\\
&\partial_t \left(\frac{1}{2}|v|^2+\hat{e}(\xi,\theta)\right)-\partial_{\alpha}\left(\sit v_i\right)=r, \nonumber
\end{align}
and satisfies the entropy inequality
\begin{align}
\label{entr.in.ws}
\partial_t\hat{\eta} (\xi, \theta) \geq\frac{r}{\theta}.
\end{align}
in the sense of distributions.

From now on we restrict to weak entropic solutions of the augmented system (\ref{aug.sys.ws}) :
\begin{definition} \label{def.entr.ws}
	A weak solution to (\ref{aug.sys.ws}) consists of a set of functions $\xi:=(F,\zeta,w)\in L^{\infty}(L^p)\times L^{\infty}(L^q)\times L^{\infty}(L^r),$
	$v\in L^{\infty}(L^2)$ and $\theta\in L^{\infty}(L^{\ell}),$ with
        $$
        \Sigma_{i \alpha} =  \sit \in L^1(Q_T)  \, , \quad \sit v_i \in L^1 (Q_T)
        $$
	which satisfy (\ref{aug.sys.ws}) in the sense of 
	distributions. The solution is said to be entropy weak solution, if in addition, satisfies the inequality (\ref{entr.in.ws}) in the sense of distributions.
\end{definition}
We require the exponents $p\geq 4,$ because, in view of Lemma \ref{lemma.weak.minors}, the quantities $\mathrm{cof}F$, $\det F$ 
satisfy the weak form of the conservation laws \eqref{geomcons}, what allows to extend an entropy weak solution 
of \eqref{sys1adiab}, \eqref{entr.in} to a weak solution of the augmented system \eqref{aug.sys.ws}
satisfying the entropy inequality \eqref{entr.in.ws}.

As before, a strong solution (typically in $W^{1,\infty}$) of (\ref{aug.sys.ws})  satisfies the entropy identity
\begin{align}
\label{entr.eq.ws}
\partial_t\hat{\eta}=\frac{r}{\theta}.
\end{align}
Consider an entropy weak solution $(\xi,v,\theta)^T$ and a strong solution $(\bar{\xi},\bar{v},\bar{\theta})^T$ and write the difference of the weak form of
equations (\ref{aug.sys.ws}) to obtain the following three integral identities
\begin{align}
\int(\xi^B-\bar{\xi}^B)(x,0)\phi_1(x,0)\:dx &+\int_0^T \int(\xi^B-\bar{\xi}^B)\partial_t\phi_1\:dx\:dt 
\nonumber
\\
&=\int_0^T \int \left(\phv-\phbv\right)\partial_{\alpha}\phi_1\:dx\:dt,
\label{eq1.wk}
\end{align}

\begin{align}
\int(v_i-\bar{v_i})&(x,0)\phi_2(x,0)\:dx +\int_0^T\int (v_i-\bar{v_i})\partial_t\phi_2\:dx\:dt
\nonumber \\
&=\int_0^T\int \left(\sit-\sitb\right)\partial_{\alpha}\phi_2\:dx\:dt\;,
\label{eq2.wk}
\end{align}
and
	\begin{align}
	\label{eq3.wk}
	\int &\left(\frac{1}{2}|v|^2+\hat{e}-\frac{1}{2}|\bar{v}|^2-\bar{\hat{e}}\right)(x,0)\phi_3(x,0)\:dx +
	\int_0^T \int \left(\frac{1}{2}|v|^2+\hat{e}-\frac{1}{2}|\bar{v}|^2-\bar{\hat{e}}\right)\partial_t\phi_3\:dx\:dt\nonumber\\
	&=\int_0^T \int \left(\sit v_i-\sitb\bar{v_i}\right)\partial_{\alpha}\phi_3\:dx\:dt \nonumber\\
	&- \int_0^T \int (r-\bar{r})\phi_3\:dx\:dt,
	\end{align}
for any $\phi_i\in C^1_c(Q_T)$, $i=1, 2, 3$. Similarly, testing the difference of (\ref{entr.in.ws}) and (\ref{entr.eq.ws}) against $\phi_4\in C^1_c(Q_T),$  with $\phi_4\ge 0$, we have
\begin{align}
\label{entr1.wka}
-\int(\hat{\eta}-\bar{\hat{\eta}})(x,0)\phi_4(x,0)\:dx-  \int_0^T\int(\hat{\eta}-\bar{\hat{\eta}})\partial_t\phi_4\:dx\:dt\geq \int_0^T\int\left(\frac{r}{\theta}-\frac{\bar{r}}{\bar{\theta}}\right)\phi_4\:dx\:dt.
\end{align}
We then choose $(\phi_1,\phi_2,\phi_3)=-\bar\theta\,G(\bar{U})\varphi(t)=( -\pgb ,-\bar{v}_i,1)^T \varphi(t)$, for some $\varphi\in C_0^1\big ([0,T) \big )$, thus (\ref{eq1.wk}), (\ref{eq2.wk}) and (\ref{eq3.wk}) become
\begin{small}
	\begin{align}
	\label{eq1phi.wk}
	\int &\left(-\pgb(\xi^B-\bar{\xi}^B)\right)(x,0)\varphi(0)\:dx +\int_0^T\int  -\pgb(\xi^B-\bar{\xi}^B)\varphi^{\prime}(t)\:dx\:dt \nonumber\\
	&=-  \int_0^T\int \left[\partial_t\Big(-\pgb\Big)(\xi^B-\bar{\xi}^B)+\partial_{\alpha}\Big(\pgb\Big)\Big(\phv-\phbv\Big)\right]\varphi dxdt \, ,
	\end{align}
\end{small}
\begin{align}
\label{eq2phi.wk}
\int(-\bar{v_i}&(v_i-\bar{v_i}))(x,0)\varphi(0)\:dx +\int_0^T\int  -\bar{v_i}(v_i-\bar{v_i})\varphi^{\prime}(t)\:dx\:dt\nonumber\\
&=-\int_0^T\int\left[-\partial_{\alpha}\left(\sitb\right)(v_i-\bar{v_i})\right.\nonumber\\
&\qquad\left.+\partial_{\alpha}\bar{v_i}  \left(\sit-\sitb\right)   \right]\varphi\:dx\:dt\;,
\end{align} 
and
\begin{align}
\label{eq3phi.wk}
\int\left(\frac{1}{2}|v|^2+\hat{e}-\frac{1}{2}|\bar{v}|^2-\bar{\hat{e}}\right)(x,0)\varphi(0)\:dx +& 
\int_0^T\int \left(\frac{1}{2}|v|^2+\hat{e}-     \frac{1}{2}|\bar{v}|^2-\bar{\hat{e}}\right)\varphi^{\prime}(t)\:dx\:dt\nonumber\\  
&=-\int_0^T\int (r-\bar{r})\varphi(t)\:dx\:dt.
\end{align}

For inequality (\ref{entr1.wka}), we choose accordingly $\phi_4:=\bar{\theta}\varphi(t)\geq 0$, $\varphi \geq 0$ so that
\begin{align}
\label{entr2.wk}
-\int(\bar{\theta}(\hat{\eta}-\bar{\hat{\eta}}))(x,0)\varphi(0)\:dx-& \int_0^T\int  \bar{\theta}(\hat{\eta}-\bar{\hat{\eta}})\varphi^{\prime}(t)\:dx\:dt\nonumber\\
&\geq \int_0^T \int \left[\partial_t\bar{\theta}(\hat{\eta}-\bar{\hat{\eta}})+\bar{\theta}\left(\frac{r}{\theta}-\frac{\bar{r}}{\bar{\theta}}\right)\right]\varphi(t)\:dx\:dt.
\end{align}
Adding together (\ref{eq1phi.wk}), (\ref{eq2phi.wk}), (\ref{eq3phi.wk}) and (\ref{entr2.wk}) and recalling (\ref{I}), we arrive to the integral relation
\begin{align}
\label{rel.entr.wk}
\int &I(\xi,v,\theta|\bar{\xi},\bar{v},\bar{\theta})(x,0)\varphi(0)\:dx+ \int_0^T \int I(\xi,v,\theta|\bar{\xi},\bar{v},\bar{\theta})\varphi^{\prime}(t)\:dx\:dt \nonumber \\ 
&\!\!\!\geq-  \!\!\int_0^T  \int \!\Big[\!\!-\partial_t\bar{\theta}\hat{\eta}(\xi,\theta|\bar{\xi},\bar{\theta})+\partial_t\bar{\xi}^B \frac{\partial \hat{\psi}}{\partial \xi^B}(\xi,\theta|\bar{\xi},\bar{\theta})
\nonumber\\
&\; +\partial_{\alpha}\left(\pgb\right)\!\Big(\ph-\phb\Big)\!(v_i-\bar{v}_i) \nonumber\\
&\; +\partial_{\alpha}\bar{v}_i\Big(\ph-\phb\Big) \Big(\pg-\pgb\Big)
+(\theta-\bar{\theta})\left(\frac{r}{\theta}-\frac{\bar{r}}{\bar{\theta}}\right)\Big]\varphi(t) \:dx\:dt.
\end{align}
From here on we consider the case that the terms  $r=0$ and $\bar{r}=0$ and use (\ref{rel.entr.wk}), and proceed to prove recovery of smooth solutions from entropic weak solutions. The result is stated below:
\begin{theorem}
	\label{Uniqueness_thm}
	Let $\bar{U}$ be a Lipschitz bounded solution of (\ref{aug.sys.ws}) subject to the constitutive theory (\ref{con.rel.aug}) in $Q_T$ with initial data $\bar{U}^0$ and $U$ be an entropy weak solution of 
	(\ref{aug.sys.ws})-(\ref{entr.in.ws}), (\ref{con.rel.aug}) with initial data $U^0.$ Suppose that $\nabla_{\xi}^2\hat{\psi}(\xi,\theta)>0$ and $\hat{\eta}_{\theta}(\xi,\theta)>0$ and the growth
	conditions (\ref{gr.con.1}), (\ref{gr.con.2}), (\ref{gr.con3}) hold for $p\geq 4,$ and $q,r\geq 2.$ If $\bar{U}\in\Gamma_{M,\delta},$ for all $(x,t)\in Q_T$, and some $M,\delta>0,$ then there exist constants $C_1$ and $C_2$ such that
	\begin{align}
	\label{est.wk}
	\int I(U(t)|\bar{U}(t))\:dx\leq C_1e^{C_2t}\int I(U^0|\bar{U}^0)\:dx\;,
	\end{align}
	for any $t\in [0,T]$.
	Thus, in particular, whenever $U_0=\bar U_0$ a.e. then $U=\bar U$ a.e.
\end{theorem}

\begin{proof} 
Let $\{\varphi_n\}$ be a sequence of monotone decreasing functions such that $\varphi_n\in C_0^1([0,T))$, $\varphi_n\geq 0,$ for all $n\in\mathbb{N},$ converging as $n \to \infty$ to the Lipschitz function
\begin{align*}
\varphi(\tau)=\begin{dcases}
1 & 0\leq\tau\leq t\\
\frac{t-\tau}{\varepsilon}+1 & t\leq\tau\leq t+\varepsilon\\
0 & \tau\geq t+\varepsilon
\end{dcases}
\end{align*}
for some $\varepsilon\geq 0$ and take the integral relation~\eqref{rel.entr.wk} against the functions $\varphi_n$:
\begin{align}
\label{rel.entr.wk-n}
\int& I(\xi,v,\theta| \bar{\xi},\bar{v},\bar{\theta})(x,0)\varphi_n(0)  \:dx  + \int_0^T \int I(\xi,v,\theta|\bar{\xi},\bar{v},\bar{\theta})\varphi^{\prime}_n(\tau)\:dx\:d\tau \nonumber \\ 
&\!\!\!\geq- \!\!\int_0^T  \int \!\left[-\partial_t\bar{\theta}\hat{\eta}(\xi,\theta|\bar{\xi},\bar{\theta})
+ \partial_t\bar{\xi}^B\frac{\partial \hat{\psi}}{\partial \xi^B}(\xi,\theta|\bar{\xi},\bar{\theta})
\right.\nonumber\\
&\;\;  +\partial_{\alpha}\left(\pgb\right)\!\left(\ph-\phb\right)\!(v_i-\bar{v}_i)\nonumber\\
&\;\;\left.+\partial_{\alpha}\bar{v}_i\left(\ph-\phb\right)\left(\pg-\pgb\right)
\right]\varphi_n(\tau) \:dx\:d\tau \, .
\end{align}
Passing to the limit as $n\to \infty$ we get
\begin{align*}
\int I(\xi,v,\theta|\bar{\xi},\bar{v},\bar{\theta})(x,0)\:&dx -\frac{1}{\varepsilon}\int_{t}^{t+\varepsilon}
\int I(\xi,v,\theta|\bar{\xi},\bar{v},\bar{\theta})\:dx\:d\tau \nonumber \\ 
&\!\!\!\geq-  \!\!\int_0^{t+\varepsilon}\!\! \int \!\left[-\partial_t\bar{\theta}\hat{\eta}(\xi,\theta|\bar{\xi},\bar{\theta})
+\partial_t\bar{\xi}^B\frac{\partial \hat{\psi}}{\partial \xi^B}(\xi,\theta|\bar{\xi},\bar{\theta})
\right.
\\
&\qquad \qquad
+\partial_{\alpha}\left(\pgb\right)\!\left(\ph-\phb\right)\!(v_i-\bar{v}_i)\nonumber\\
&\;\;\left.+\partial_{\alpha}\bar{v}_i\left(\ph-\phb\right)\left(\pg-\pgb\right)
 \right] \:dx\:d\tau
\end{align*}
and then using the estimates (\ref{bound5}), (\ref{bound6}), (\ref{bound7}) and (\ref{bound8}), as $\varepsilon \to 0^{+}$, we arrive at
\begin{align*}
\int  I(\xi,v,\theta|\bar{\xi},\bar{v},\bar{\theta})\:dx\:dt 
\leq C  \int_0^t\int I(\xi,v,\theta|\bar{\xi},\bar{v},\bar{\theta})\:dx\:d\tau
+\int I(\xi,v,\theta|\bar{\xi},\bar{v},\bar{\theta})(x,0)\:dx
\end{align*}
for $t\in (0,T);$
which in turn implies (\ref{est.wk}) by Gronwall's inequality. 
\end{proof}

\section{ Convergence from entropy weak solutions of thermoelasticity to strong solutions of adiabatic thermoelasticity}\label{sec7}

As already noted in \cite{MR1831175} the problems on asymptotic limits are quite connected to weak-strong uniqueness
results for weak (or for measure-valued) solutions. We illustrate this point in the present section by performing a convergence
from entropy weak solutions of the thermoelasticity system with Fourier heat conduction to strong solutions of
the adiabatic thermoelasticity system.

Let  $(F, v, \theta)^T$ be a weak solution of the system of thermoelasticity 
\begin{align}
\begin{split}
\label{ThE}
\partial_t F_{i \alpha} - \partial_{\alpha} v_i &=0\\
\partial_t v_i-\partial_{\alpha}\left(   \frac{\partial \hat{\psi}}{\partial \xi^B}  \Big ( \Phi(F),\theta \Big ) \frac{\partial\Phi^B}{\partial F_{i\alpha}}(F) \right)&=0\\
\partial_t \left(\frac{1}{2}|v|^2+\hat{e}(\Phi(F),\theta)\right)
-\partial_{\alpha}\left(  \frac{\partial \hat{\psi}}{\partial \xi^B}  \Big ( \Phi(F),\theta \Big ) \frac{\partial\Phi^B}{\partial F_{i\alpha}}(F)   v_i\right)&=\partial_{\alpha}(k\partial_{\alpha}\theta) + r
\\
\partial_{\alpha}F_{i\beta} - \partial_{\beta}F_{i\alpha} &=0\\
\end{split}
\end{align}
that satisfies in the sense of distributions the weak form of the entropy inequality
\begin{equation}
\label{ThE.entropy}
\partial_t\hat{\eta} \big(\Phi(F),\theta\big) -\partial_{\alpha}\left(k\frac{\nabla\theta}{\theta}\right)\geq k\frac{|\nabla\theta|^2}{\theta^2}+\frac{r}{\theta}.
\end{equation}

Lemma \ref{lemma.weak.minors}  shows that for $y$ with sufficient integrability properties ($p \ge 4$ for dimension $d=3$)
the weak form of the transport equations \eqref{geomcons} is satisfied. Hence, 
the extended function  $(\xi, v, \theta)^T$, $\xi = \Phi (F)$, is a weak solution the augmented thermoelasticity system 
\begin{align}
\begin{split}
\label{aug.ThE}
\partial_t \xi^B-\partial_{\alpha}\left(\phv \right)&=0\\ 
\partial_t v_i-\partial_{\alpha}\left(\pg\ph\right)&=0\\
\partial_t \left(\frac{1}{2}|v|^2+\hat{e}\right)-\partial_{\alpha}\left(\pg\ph v_i\right)&=\partial_{\alpha}(k\partial_{\alpha}\theta) + r
\\
\partial_{\alpha}F_{i\beta} - \partial_{\beta}F_{i\alpha} &=0
\end{split}
\end{align}
and satisfies the inequality
\begin{equation}
\label{aug.ThE.ent}
\partial_t\hat{\eta}-\partial_{\alpha}\left(k\frac{\nabla\theta}{\theta}\right) \geq k\frac{|\nabla\theta|^2}{\theta^2}+\frac{r}{\theta}.
\end{equation}
in the sense of distributions.

Henceforth, we restrict to weak entropic solutions of \eqref{aug.ThE}, \eqref{aug.ThE.ent}; these will be functions
$U=(\xi,v,\theta)^T,$ with regularity
$$
\xi:=(F,\zeta,w)\in L^{\infty}(L^p)\times L^{\infty}(L^q)\times L^{\infty}(L^r), \quad 
v\in L^{\infty}(L^2) , \quad \theta\in L^{\infty}(L^{\ell}),
$$ 
$$
\begin{aligned}
       &\Sigma_{i \alpha} :=  \sit \in L^1(Q_T)  \, , \quad \Sigma_{i \alpha}  v_i \in L^1 (Q_T) \, , \quad k \nabla \theta \in L^1 (Q_T) 
\\
       & k \frac{\nabla \theta}{\theta} \in L^1 (Q_T) \, , \quad k \frac{ |\nabla \theta|^2}{\theta^2}  \in L^1 (Q_T) \, , \quad  \frac{r}{\theta} \in L^1 (Q_T)
\end{aligned}
$$
that satisfy  (\ref{aug.ThE}) in the sense of distributions and the inequality \eqref{aug.ThE.ent} again in the sense of distributions.

 It is natural to consider the class of entropic weak solutions,
since, in general, smooth solutions of thermoelasticity can break down in finite time due to the formation of shocks. 
Stability analysis similar to Theorem \ref{Ad.Th-el} can be performed between an entropy weak solution $U = (\xi, v, \theta)^T$
of (\ref{aug.ThE}), \eqref{aug.ThE.ent} and a strong conservative solution $\bar U = ( \bar \xi , \bar v, \bar \theta)^T$
of the augmented adiabatic thermoelasticity system 
\eqref{aug.sys.ws}, \eqref{entr.eq.ws}. The main point is to derive the relative entropy identity \eqref{rel.en.id} 
under the regularity framework of  weak entropic solutions of thermoelasticity. 
The reader should note that the formal calculations of Section \ref{sec3} leading to \eqref{rel.en.id} are performed for smooth solutions.
However, a close survey of the calculations indicates that, when dealing with entropy weak solutions,  
it would suffice to derive the identity \eqref{iden.interm}  in the sense of distributions for entropy weak solutions.
The reason is that the remainder of the derivation of \eqref{rel.en.id} 
requires only algebraic manipulations involving the strong solution $\bar U$ that can be directly performed. 
The derivation of the weak form of \eqref{iden.interm} is accomplished by an argument similar to the proof of 
\eqref{rel.entr.wk} in Section \ref{sec6}. We conclude:

\begin{theorem}
	\label{wk.ThE}
	Let $U_{k}$ be an entropy weak solution of the system (\ref{aug.ThE}),(\ref{aug.ThE.ent}), with $r=0$, subject to the constitutive relations (\ref{con.rel.aug}) with data $U_k^0$ and
	defined on $Q_{T^*}$. Let $\bar{U}=(\bar\xi,\bar v,\bar\theta)^T$ be a smooth solution of the system of adiabatic thermoelasticity 
	(\ref{aug.sys.2}),(\ref{aug.sys.2.e}),(\ref{con.rel.aug}) with $\bar{r}=0$ and data $\bar{U}^0$ defined on $\bar{Q}_T$, with $0<T<T^*$. 
	Suppose that $\nabla_{\xi}^2\hat{\psi}(\xi,\theta)>0$, $\hat{\eta}_{\theta}(\xi,\theta)>0$ and the growth conditions 
	(\ref{gr.con.1}), (\ref{gr.con.2}), (\ref{gr.con3}) hold for $p\geq 4,$ and $q,r\geq 2.$  
	If $\bar{U}\in\Gamma_{M,\delta},$ for some $M,\delta>0,$ then there exists a constant $C=C(T)$ such that
		\begin{align*} 
		\int I(U_{k}(t)|\bar{U}(t))\:dx\leq C\left(\int I(U_k^0|\bar{U}^0)\:dx+\int_{0}^{T}\int k\frac{|\nabla\bar{\theta}(s)|^2}{\bar{\theta}(s)}\;dxdt\right)\;,
		\end{align*}
	for $t\in[0,T]$. If
	\begin{align*}
    |k(\xi,\theta)|\leq C k_0|\hat{e}(\xi,\theta)|,
	\end{align*}
	for some positive constant $k_0,$ and the data satisfy
	$\lim_{k_0\to 0^+}\int I(U_{k}^0|\bar{U}^0)\:dx=0$,
	then
	\begin{align*}
	\sup_{t\in[0,T)}\int I(U_{k}(t)|\bar{U}(t))\:dx\to 0, \quad\text{as $k_0\to 0^+.$}
	\end{align*}
\end{theorem}
\begin{proof} The proof follows Theorem \ref{Ad.Th-el}. The only difference is that the derivation of the relative entropy identity 
and (\ref{Gronwal.adiab.}) need to be performed in the class of entropy weak solutions of (\ref{aug.ThE}). This is done  following the ideas used in Section \ref{sec6}.
\end{proof}

\appendix
\section{Growth conditions and auxiliary estimates} \label{AppA}

In this appendix, we prove under the growth conditions~\eqref{gr.con.1}--\eqref{gr.con.2} on the extended variables $\hat{e}$ and $\hat{\psi}$ useful bounds on the relative entropy. These bounds are employed to prove the theorems in Sections~\ref{sec4}--\ref{sec7}.

For convenience, let us recall the growth conditions
\begin{equation}
\label{Agr.con.1}
c(|\xi|_{p,q,r}+\theta^{\ell})-c\leq\hat{e}(\xi,\theta)\leq c(|\xi|_{p,q,r}+\theta^{\ell})+c
\end{equation}
\begin{equation}
\label{Agr.con.2}
\lim_{|\xi|_{p,q,r}+\theta^{\ell}\to\infty}\frac{|\partial_F\hat{\psi}|+|\partial_{\zeta}\hat{\psi}|^{\frac{p}{p-1}}+|\partial_w\hat{\psi}|^{\frac{p}{p-2}}}{|\xi|_{p,q,r}+\theta^{\ell}}=0
\end{equation}
and
\begin{equation}
\label{Agr.con3}
\lim_{|\xi|_{p,q,r}+\theta^{\ell}\to\infty}\frac{|\partial_{\theta}\hat{\psi}|}{|\xi|_{p,q,r}+\theta^{\ell}}=0
\end{equation}
for some constant $c>0$ and $p \ge 4,$ $q,r,\ell>1.$

First we prove the following lemma.
\begin{lemma}
	\label{lemma1}
	Assume that $(\bar{\xi},\bar{v},\bar{\theta})\in\Gamma_{M,\delta}$ 
	defined in \eqref{defGamma} and let $\hat{\psi}=\hat{e}-\theta\hat{\eta}\in C^3.$ Then
	\begin{enumerate}
		\item
		There exist $R=R(M,\delta)$ and a constant $C>0$ such that
			\begin{align}
			\label{bound1}
			\begin{split}
			\Big|\big(\ph-&\phb\big)\!\left(\pg-\pgb\right)\Big|\le\\
			&\leq
			\begin{dcases}
			C(|\xi|_{p,q,r}+\theta^{\ell}),  &|\xi|_{p,q,r}+\theta^{\ell}>R\\
			C(|F-\bar{F}|^2+|\zeta-\bar{\zeta}|^2+|w-\bar{w}|^2 +|\theta-\bar{\theta}|^2 ),  &|\xi|_{p,q,r}+\theta^{\ell}\leq R
			\end{dcases}
			\end{split}
			\end{align}
		for all $(\bar{\xi},\bar{v},\bar{\theta})\in\Gamma_{M,\delta}$.
		\item
		There exist $R=R(M,\delta)$ and a constant $C>0$ such that
		\begin{small}
			\begin{align}
			\label{bound2}
			\begin{split}
			&\left|\frac{\partial \hat{\psi}}{\partial \xi}(\xi,\theta|\bar{\xi},\bar{\theta})\right|
			\leq\begin{dcases}
			C(|\xi|_{p,q,r}+\theta^{\ell}), \qquad\qquad\qquad\qquad\qquad\qquad\quad |\xi|_{p,q,r}+\theta^{\ell}>R\\
			C(|F-\bar{F}|^2+|\zeta-\bar{\zeta}|^2+|w-\bar{w}|^2 +|\theta-\bar{\theta}|^2),\: |\xi|_{p,q,r}+\theta^{\ell}\leq R
			\end{dcases}
			\end{split}
			\end{align}
		\end{small}
		for all $(\bar{\xi},\bar{v},\bar{\theta})\in\Gamma_{M,\delta}$.
		\item
		There exist $R=R(M,\delta)$ and a constant $C>0$ such that
		\begin{small}
			\begin{align}
			\label{bound3}
			\begin{split}
			&\left|\left(\ph-\phb\right)(v_i-\bar{v}_i)\right|\\
			&\qquad \:\leq\begin{dcases}
			C\left(|\xi|_{p,q,r}+\theta^{\ell}+\frac{|v-\bar{v}|^2}{2}\right), &|\xi|_{p,q,r}+\theta^{\ell}+|v|^2>R\\
			C(|F-\bar{F}|^2+|\zeta-\bar{\zeta}|^2+|w-\bar{w}|^2 +|\theta-\bar{\theta}|^2 + |v-\bar{v}|^2    ), & |\xi|_{p,q,r}+\theta^{\ell}+|v|^2\leq R
			\end{dcases}
			\end{split}
			\end{align}
		\end{small}
		for all $(\bar{\xi},\bar{v},\bar{\theta})\in\Gamma_{M,\delta}$.
	\end{enumerate}
\end{lemma}
\begin{proof}
Choose $R$ sufficiently large so that $\Gamma_{M,\delta}\subset B_R:=\{(\xi,v,\theta)::\; |\xi|_{p,q,r}+\theta^{\ell}+|v|^2\leq R\}.$
Moreover, we define the compact set
\begin{equation*}
\hat{\Gamma}_{M,\delta}:=\left\{(\bar{\xi},\bar{\theta}):: |\bar{F}|\leq M, |\bar{\zeta}|\leq M, |\bar{w}|\leq M,\;\; 0<\delta\leq\bar{\theta}\leq M\right\},
\end{equation*}
for which, then, also holds $\hat{\Gamma}_{M,\delta}\subset \hat{B}_R:=\{(\xi,\theta)::\; |\xi|_{p,q,r}+\theta^{\ell}\leq R\}.$
We divide the proof into 3 steps.

\textbf{Step $1.$} To prove the first assertion, we first consider the case $|\xi|_{p,q,r}+\theta^{\ell}>R.$ Using Young's inequality and assumption (\ref{Agr.con.2}), we have
for $(\bar{\xi},\bar{\theta})\in\hat{\Gamma}_{M,\delta}$
\begin{align*}
&\left|\left(\ph-\phb\right)\!\left(\pg-\pgb\right)\right|\\
&\qquad\quad\leq C_1(1+|\partial_F\hat{\psi}|+(1+|\partial_{\zeta}\hat{\psi}|)\cdot(1+|F|)+(1+|\partial_w\hat{\psi}|)\cdot(1+|F|^2))\\
&\qquad\quad\leq C_2(|\xi|_{p,q,r}+\theta^{\ell})+C_3\left(|\partial_F\hat{\psi}|+|\partial_{\zeta}\hat{\psi}|^{\frac{p}{p-1}}+|\partial_w\hat{\psi}|^{\frac{p}{p-2}}+1\right)\\
&\qquad\quad\leq C_4(|\xi|_{p,q,r}+\theta^{\ell}+1)\\
&\qquad\quad\leq C(|\xi|_{p,q,r}+\theta^{\ell}).
\end{align*}
On the domain  $|\xi|_{p,q,r}+\theta^{\ell}\leq R$, 
\begin{align*}
\left| \left(\ph-\phb\right)\! \right.& \left.\left(\pg-\pgb\right)\right|\\
&
\leq C(|F-\bar{F}|^2+|\zeta-\bar{\zeta}|^2+|w-\bar{w}|^2 + |\theta-\bar{\theta}|^2 ),
\end{align*}
where $C=\displaystyle\max_{\hat{B}_R}\left\{\nabla_{(\xi,\theta)}\:\;\left(\pg\right),\nabla_F \left(\ph\right)\right\}$.

\textbf{Step $2.$}  Using (\ref{partial.g.rel}), Young's inequality and (\ref{Agr.con.2}), we obtain
\begin{align*}
\left|\frac{\partial \hat{\psi}}{\partial \xi}(\xi,\theta|\bar{\xi},\bar{\theta})\right|
&\leq |\partial_F\hat{\psi}|+|\partial_{\zeta}\hat{\psi}|+|\partial_w\hat{\psi}|+C_1(|F|+|\zeta|+|w|+\theta)+C_2\\
&\leq C_3\left(|\partial_F\hat{\psi}|+|\partial_{\zeta}\hat{\psi}|^{\frac{p}{p-1}}+|\partial_w\hat{\psi}|^{\frac{p}{p-2}}\right)
+C_4(|\xi|_{p,q,r}+\theta^{\ell}) +C_5 \\
&\leq C(|\xi|_{p,q,r}+\theta^{\ell})
\end{align*}
for $|\xi|_{p,q,r}+\theta^{\ell}>R$ and $(\bar{\xi},\bar{\theta})\in\hat{\Gamma}_{M,\delta}.$
In the complementary region $|\xi|_{p,q,r}+\theta^{\ell}\leq R$ 
\begin{align*}
\left|\frac{\partial \hat{\psi}}{\partial \xi}(\xi,\theta|\bar{\xi},\bar{\theta})\right|
\leq C(|F-\bar{F}|^2+|\zeta-\bar{\zeta}|^2+|w-\bar{w}|^2+|\theta-\bar{\theta}|^2),
\end{align*}
for $C=\displaystyle\max_{\hat{B}_R}\left|\nabla_{(\xi,\theta)}^3\hat{\psi}(\xi,\theta)\right|.$

\textbf{Step $3.$} To prove (\ref{bound3}), we proceed in a similar fashion. First, for $|\xi|_{p,q,r}+\theta^{\ell} +|v|^2>R$ and 
$(\bar{\xi},\bar{v},\bar{\theta})\in\Gamma_{M,\delta}$, we have
\begin{align*}
\left|\left(\ph-\phb\right)(v_i-\bar{v}_i)\right|&\leq\
\tfrac{1}{2} \left| \frac{\partial\Phi^B}{\partial F}(F)-\frac{\partial\Phi^B}{\partial F}(\bar{F}) \right|^2+ \tfrac{1}{2} |v-\bar{v}|^2\\
&\leq\frac{C_1}{2}(1+|F|+|F|^2)^2+\frac{|v-\bar{v}|^2}{2}\\
&\leq C_3+|F|^p+\frac{|v-\bar{v}|^2}{2}\\
&\leq C\left(|\xi|_{p,q,r}+\theta^{\ell}+\frac{|v-\bar{v}|^2}{2}\right)\;.
\end{align*}
Then for $|\xi|_{p,q,r}+\theta^{\ell}+|v|^2\leq R$ and for all $(\bar{\xi},\bar{v},\bar{\theta})\in\Gamma_{M,\delta}$, we also get
\begin{align*}
\left|\ph-\phb\right|^2 &\leq C_1(|F-\bar{F}|)^2\\
&\leq C(|F-\bar{F}|^2+|\zeta-\bar{\zeta}|^2+|w-\bar{w}|^2 +|\theta-\bar{\theta}|^2),
\end{align*}
where $C_1=\displaystyle\max_{B_R}\nabla_F\left(\ph\right)$. \end{proof}

Using the results of Lemma \ref{lemma1}, we can now bound $I(\xi,v,\theta|\bar{\xi},\bar{v},\bar{\theta})$ as the following lemma indicates. 
\begin{lemma}
	\label{lemma2}
	Under the assumptions of Lemma \ref{lemma1},
	\begin{enumerate}
		\item
		There exist $R=R(M,\delta)$ and constants $K_1=K_1(M,\delta,c)>0,\:K_2=K_2(M,\delta,c)>0$ such that
		\begin{small}
			\begin{align}
			\label{bound4}
			\begin{split}
			&I(\xi,v,\theta|\bar{\xi},\bar{v},\bar{\theta})\geq\\
			&\:\geq
			\begin{dcases}
			\frac{K_1}{2}(|\xi|_{p,q,r}+\theta^{\ell}+|v|^2), \qquad\qquad\qquad\quad\quad\quad\quad\quad\quad\quad |\xi|_{p,q,r}\!+\theta^{\ell}\!+|v|^2>R\\
			K_2(|F-\bar{F}|^2\!+\!|\zeta-\bar{\zeta}|^2\!+\!|w-\bar{w}|^2\!+\!|\theta-\bar{\theta}|^2\!+\!|v-\bar{v}|^2),\: |\xi|_{p,q,r}\!+\theta^{\ell}\!+|v|^2\leq R
			\end{dcases}
			\end{split}
			\end{align}
		\end{small}
		for all $(\bar{\xi},\bar{v},\bar{\theta})\in\Gamma_{M,\delta}$.
		\item
		There exists a constant $C>0$ such that
		\begin{align}
		\label{bound5}
		|\hat{\eta}(\xi,\theta|\bar{\xi},\bar{\theta})|\leq C I(\xi,v,\theta|\bar{\xi},\bar{v},\bar{\theta})
		\end{align}
		for all $(\bar{\xi},\bar{v},\bar{\theta})\in\Gamma_{M,\delta}$.
		\item
		There exist constants $K_1,\:K_2$ such that
		\begin{small}
			\begin{align}
			\label{bound9}
			\begin{split}
			&I(\xi,v,\theta|\bar{\xi},\bar{v},\bar{\theta})\geq\\
			&\:\geq
			\begin{dcases}
			\frac{K_1}{4}(|\xi-\bar{\xi}|_{p,q,r}+|\theta-\bar{\theta}|^{\ell}+|v-\bar{v}|^2), \qquad\quad\quad\quad\quad\quad\; |\xi|_{p,q,r}\!+\theta^{\ell}\!+|v|^2>R\\
			K_2(|F-\bar{F}|^2\!+\!|\zeta-\bar{\zeta}|^2\!+\!|w-\bar{w}|^2\!+\!|\theta-\bar{\theta}|^2\!+\!|v-\bar{v}|^2),\: |\xi|_{p,q,r}\!+\theta^{\ell}\!+|v|^2\leq R
			\end{dcases}
			\end{split}
			\end{align}
		\end{small}
		for all $(\bar{\xi},\bar{v},\bar{\theta})\in\Gamma_{M,\delta}$.
		\item
		There exist constants $C_1,C_2,C_3>0$ such that
		\begin{equation}
		\label{bound6}
		\left|\left(\ph-\phb\right)\!\!\left(\pg-\pgb\right)\right|\leq C_1 I(\xi,v,\theta|\bar{\xi},\bar{v},\bar{\theta})
		\end{equation}
		\begin{equation}
		\label{bound7}
		\left|\frac{\partial \hat{\psi}}{\partial \xi}(\xi,\theta|\bar{\xi},\bar{\theta})\right|\leq C_2 I(\xi,v,\theta|\bar{\xi},\bar{v},\bar{\theta})
		\end{equation}
		and \begin{equation}
		\label{bound8}
		\left|\left(\ph-\phb\right)(v_i-\bar{v}_i)\right|\leq C_3 I(\xi,v,\theta|\bar{\xi},\bar{v},\bar{\theta})
		\end{equation}
		for all $(\bar{\xi},\bar{v},\bar{\theta})\in\Gamma_{M,\delta}$.
	\end{enumerate}
\end{lemma}
\begin{proof}
Let $(\bar{\xi},\bar{v},\bar{\theta})\in\Gamma_{M,\delta}$ and choose $r=r(M):=M^p+M^q+M^r+M^{\ell}+M^2$
for which $\Gamma_{M,\delta}\subset B_r=\{(\xi,v,\theta)::\; |\xi|_{p,q,r}+\theta^{\ell}+|v|^2\leq r\}$. 
We divide the proof into four steps.

\textbf{Step $1.$} Note that  $I(\xi,v,\theta|\bar{\xi},\bar{v},\bar{\theta})$ can be written in the form
\begin{equation*}
I(\xi,v,\theta|\bar{\xi},\bar{v},\bar{\theta})=\hat{e}-\bar{\hat{\psi}}-\partial_F\hat{\psi}(F-\bar{F})-\partial_{\zeta}\hat{\psi}(\zeta-\bar{\zeta})-\partial_w\hat{\psi}(w-\bar{w})-\hat{\eta}\bar{\theta}+\frac{1}{2}|v-\bar{v}|^2 \, .
\end{equation*}
Consider first the region $|\xi|_{p,q,r}+\theta^{\ell}+|v|^2>R,$ for $R>r(M)+1$. 
Using Young's inequality and (\ref{Agr.con.1}), (\ref{Agr.con3}) we deduce that by selecting $R$ sufficiently large
we obtain the bound
\begin{align*}
I(\xi,v,\theta|\bar{\xi},\bar{v},\bar{\theta})
&\geq \min\left\{c,\frac{1}{2}\right\} (|\xi|_{p,q,r}+\!\theta^{\ell}\!+\!|v|^2)\!-\!c_1|\hat{\eta}(\xi,\theta)|\!-\!c_2(|F|\!+\!|\zeta|\!+\!|w|)\!-\!c_3|v|\!-\!c_4\\
&\quad\geq K_1(|\xi|_{p,q,r}+\theta^{\ell}+|v|^2)-c_1|\hat{\eta}(\xi,\theta)|-c_5\\
&\quad\geq \frac{K_1}{2}(|\xi|_{p,q,r}+\theta^{\ell}+|v|^2).
\end{align*}
for some constant $K_1$ and $|\xi|_{p,q,r}+\theta^{\ell}+|v|^2>R$.

In the complementary region $|\xi|_{p,q,r}+\theta^{\ell}+|v|^2\leq R$, it holds 
\begin{align*}
\frac{1}{\bar{\theta}}I(\xi,v,\theta|\bar{\xi},\bar{v},\bar{\theta})
&=\bar{\theta}\tilde{H}(A(U)|A(\bar{U}))\\
&=\bar{\theta}[\tilde{H}(A(U))-\tilde{H}(A(\bar{U}))-\tilde{H}_V(A(\bar{U}))(A(U)-A(\bar{U}))]\\
&\geq\min_{\substack{\tilde{U}\in B_R \\ \delta\leq\bar{\theta}\leq M}}\{\bar{\theta}\tilde{H}_{VV}(A(\tilde{U}))\}|A(U)-A(\bar{U})|^2,
\end{align*}
since $\tilde{H}(V)$ is convex in $V:=A(U)=(\xi,v,E)^T$. Hence,
\begin{align*}
|U-\bar{U}|&=\left|\int_0^1\frac{d}{d\tau}[A^{-1}(\tau A(U)+(1-\tau)A(\bar{U}))]\:d\tau\right|\\
&\leq\left|\int_0^1\nabla_V(A^{-1})(\tau A(U)+(1-\tau)A(\bar{U}))\:d\tau\right|\:|A(U)-A(\bar{U})|\\
&\leq C |A(U)-A(\bar{U})|,
\end{align*}
where
$$C=\sup_{\substack{U\in B_R \\ \bar{U}\in\Gamma_{M,\delta}}}\left|\int_0^1\nabla_V(A^{-1})(\tau A(U)+(1-\tau)A(\bar{U}))\:d\tau\right|\:|A(U)-A(\bar{U})|<\infty.$$
Therefore
\begin{align*}
I(\xi,v,\theta|\bar{\xi},\bar{v},\bar{\theta})\geq\frac{K_2}{C}|U-\bar{U}|^2,
\end{align*}
for $K_2:=\displaystyle\delta\min_{\tilde{U}\in B_R}\{\bar{\theta}\tilde{H}_{VV}(A(\tilde{U}))\}>0.$

\textbf{Step $2.$} Now, to prove (\ref{bound5}) for $|\xi|_{p,q,r}+\theta^{\ell} +|v|^2 >R,$ for $R>r(M)+1,$
 we use the growth assumption (\ref{gr.con3}) to get
\begin{small}
	\begin{align*}
	&\lim_{|\xi|_{p,q,r}+\theta^{\ell}\to\infty}\frac{|\hat{\eta}(\xi,\theta|\bar{\xi},\bar{\theta})|}{|\xi|_{p,q,r}+\theta^{\ell}}
	=\lim_{|\xi|_{p,q,r}+\theta^{\ell}\to\infty}\frac{|\hat{\eta}(\xi,\theta)|}{|\xi|_{p,q,r}+\theta^{\ell}}=0.
	\end{align*}
\end{small}
and
\begin{align*}
&\lim_{|\xi|_{p,q,r}+\theta^{\ell}+|v|^2\to\infty}\frac{|\hat{\eta}(\xi,\theta|\bar{\xi},\bar{\theta})|}{|\xi|_{p,q,r}+\theta^{\ell}+|v|^2} =0.
\end{align*}
On the complementary region $|\xi|_{p,q,r}+\theta^{\ell} +|v|^2\leq R,$ 
\begin{align*}
|\hat{\eta}(\xi,\theta|\bar{\xi},\bar{\theta})|&\leq\max_{ B_R} |\nabla_{(\xi,\theta)}^2\hat{\eta}|(|F-\bar{F}|^2\!+\!|\zeta-\bar{\zeta}|^2\!+\!|w-\bar{w}|^2\!+\!|\theta-\bar{\theta}|^2)\\
&\leq C I(\xi,v,\theta|\bar{\xi},\bar{v},\bar{\theta})\;,
\end{align*}
by~\eqref{bound4}.

\textbf{Step $3.$}  Since $(\bar{\xi},\bar{v},\bar{\theta})\in\Gamma_{M,\delta}\subset B_r,$ there holds
\begin{align*}
|F-\bar{F}|^p&+|\zeta-\bar{\zeta}|^q+|w-\bar{w}|^r+|\theta-\bar{\theta}|^{\ell}+|v-\bar{v}|^2\\
&\leq(|F|+M)^p+(|\zeta|+M)^q+(|w|+M)^r+(\theta+M)^{\ell}+(|v|+M)^2\;.
\end{align*}
Since
\begin{align*}
\lim_{|\xi|_{p,q,r}+\theta^{\ell}+|v|^2\to\infty}\frac{(|F|+M)^p+(|Z|+M)^q+(|w|+M)^r+(\theta+M)^{\ell}+(|v|+M)^2}{|\xi|_{p,q,r}+\theta^{\ell}+|v|^2}=1\;,
\end{align*}
we may select $R$ such that
\begin{align*}
|F-\bar{F}|^p+|\zeta-\bar{\zeta}|^q+|w-\bar{w}|^r+&|\theta-\bar{\theta}|^{\ell}+|v-\bar{v}|^2\\
&\leq2(|F|^p+|\zeta|^q+|w|^r+\theta^{\ell}+|v|^2+1)
\end{align*}
when $|\xi|_{p,q,r}\!+\theta^{\ell}\!+|v|^2 \geq R.$ Thus~(\ref{bound9}) follows.

\textbf{Step $4.$} The proof of (\ref{bound6}) and (\ref{bound7}) follows from (\ref{bound4}) and Lemma~\ref{lemma1}
by an argument similar to Step 2 using \eqref{Agr.con.2}.
It remains to show (\ref{bound8}) in the region $|\xi|_{p,q,r}+\theta^{\ell}+|v|^2>R.$
Using~(\ref{bound4}), (\ref{bound9}) and $(\bar{\xi},\bar{v},\bar{\theta})\in\Gamma_{M,\delta}$, 
\begin{align*}
\left|\left(\ph-\phb\right)(v_i-\bar{v}_i)\right|&\leq
\tfrac{1}{2} \left| \frac{\partial\Phi^B}{\partial F}(F)-\frac{\partial\Phi^B}{\partial F}(\bar{F}) \right|^2+ \tfrac{1}{2} |v-\bar{v}|^2 \\
&\le   C_1 (1+ |F|^4 ) + C_2 I(\xi,v,\theta|\bar{\xi},\bar{v},\bar{\theta}) \\
&\leq C (|\xi|_{p,q,r}+\theta^{\ell})+C_2 I(\xi,v,\theta|\bar{\xi},\bar{v},\bar{\theta})\\
&\leq C I(\xi,v,\theta|\bar{\xi},\bar{v},\bar{\theta}) \, ,
\end{align*}
which completes the proof. \end{proof}

\begin{lemma}
	\label{lemma3}
	Under the assumptions of Lemma \ref{lemma1}, and given that
	\begin{align}
	\label{assumption_theta_k}
	0<\underline{\delta}\leq\theta(x,t), \quad \text{and} \quad k(\xi,\theta)+\frac{\theta^2}{k(\xi,\theta)}\leq C' \hat{e}(\xi,\theta),
	\end{align}
	there exists a constant $C=C(\underline{\delta})>0$ such that
	\begin{align}
	\label{bound10}
	\frac{\theta^2}{k}\left|\frac{k}{\theta}-\frac{\bar{k}}{\bar{\theta}}\right|^2
	\leq C I(\xi,v,\theta|\bar{\xi},\bar{v},\bar{\theta})
	\end{align}
	for all $(\bar{\xi},\bar{v},\bar{\theta})\in\Gamma_{M,\delta}$.
\end{lemma}
\begin{proof}
Consider first the region $|\xi|_{p,q,r}+\theta^{\ell}>R.$ Using the assumptions (\ref{assumption_theta_k}) and
because of (\ref{Agr.con.1}) we have
\begin{align*}
\frac{\theta^2}{k}\left|\frac{k}{\theta}-\frac{\bar{k}}{\bar{\theta}}\right|^2
&\leq C_1 \frac{\theta^2}{k}\left(\frac{k^2}{\theta^2}+1\right)\\
&= C_1 \left(k+\frac{\theta^2}{k}\right)\\
&\leq C \hat{e}(\xi,\theta)
\end{align*}
for all $(\bar{\xi},\bar{\theta})\in\hat{\Gamma}_{M,\delta}.$ Then bound~(\ref{bound4}) yields~(\ref{bound10}) 
in this region.
In the complementary region $|\xi|_{p,q,r}+\theta^{\ell}\leq R,$ since $0<\underline{\delta}\leq\theta$ there holds
\begin{align*}
\frac{\theta^2}{k}\left|\frac{k}{\theta}-\frac{\bar{k}}{\bar{\theta}}\right|^2
\leq C (|F-\bar{F}|^2+|\theta-\bar{\theta}|^2),
\end{align*}
for all $(\bar{\xi},\bar{\theta})\in\hat{\Gamma}_{M,\delta}.$ Hence, (\ref{bound9}) implies
(\ref{bound10}).
\end{proof}

\section*{Acknowledgement}
This project has received funding from the European Union's Horizon 2020 programme under the Marie Sklodowska-Curie
grant agreement No 642768. Christoforou was partially supported by the Internal grant "SBLawsMechGeom" No 21036 from University of Cyprus. 
AET was partially supported by the Simons - Foundation grant 346300 and the Polish Government MNiSW 2015-2019 matching fund;
he thanks the Institute of Mathematics of the Polish Academy of Sciences, Warsaw, for their hospitality 
during his stay as a Simons Visiting Professor.


\end{document}